\documentclass[11pt,a4paper]{amsart}
\usepackage[T1]{fontenc}
\usepackage{iftex}
\ifPDFTeX
\usepackage[utf8]{inputenc}
\fi
\usepackage[english]{babel}
\usepackage[margin=30mm]{geometry}
\usepackage{amsmath,amssymb,amsfonts,mathtools,mathrsfs,amscd}
\usepackage{xcolor}
\usepackage[hidelinks]{hyperref}
\numberwithin{equation}{section}
\theoremstyle{plain}
\newtheorem{theorem}{Theorem}[section]
\newtheorem{lemma}[theorem]{Lemma}
\newtheorem{propos}[theorem]{Proposition}
\newtheorem{corollary}[theorem]{Corollary}
\theoremstyle{definition}
\newtheorem{definition}[theorem]{Definition}
\newtheorem{Cit}[theorem]{Quotation}
\newtheorem{remark}[theorem]{Remark}
\newtheorem{example}[theorem]{Example}

\def\ord{\operatorname{ord}}

\def\C{\mathbb{C}}
\begin{document}
\title{Geometry of holonomic modules over the first Weyl algebra}
\author[M.\,D.~Sulimov]{M.\,D.~Sulimov}
\address{Institute for Information Transmission Problems (Kharkevich Institute), Russian Academy of Sciences; Skolkovo Institute of Science and Technology}
\email{mark.game.sulimov@gmail.com}
\author[A.\,V.~Petukhov]{A.\,V.~Petukhov}
\address{Institute for Information Transmission Problems (Kharkevich Institute), Russian Academy of Sciences}
\email{alex-{}-2@yandex.ru}
\date{}
\begin{abstract}
We give a geometric description of the formal local data of holonomic modules over the first Weyl algebra, the algebra of differential operators in one variable. The exponential part of the associated formal meromorphic connection is encoded by a truncated jet of an affine plane curve at infinity, and the regular part by a regular formal meromorphic connection. We prove that these geometric data uniquely determine the `elementary' blocks in the formal decomposition of holonomic modules. The resulting correspondence is equivariant under automorphisms of the first Weyl algebra, including the Fourier transform.
\end{abstract} 

\maketitle

\noindent 
{\Small \text{\scshape{Keywords}}: Holonomic $D$-modules, formal meromorphic connections, stationary phase formula, truncated jets of affine plane curves, logarithmic families of curves}

\section{Introduction}

In this paper we take a step towards a geometric description of the Stokes data associated with a holonomic module over the first Weyl algebra, which we denote by $W_{1, \mathbb C}$.
As is well known, $W_{1, \mathbb C}$-modules correspond to systems of linear homogeneous differential equations in one variable with polynomial coefficients. 
For a differential equation, one is interested both in its solutions, formal or analytic, and in their asymptotic behavior. 
Stokes data bring these two aspects together. The category of Stokes data is of particular interest because it is equivalent to the category of analytic meromorphic connections; see~\cite{Kas15, Mal91}.

If we fix the exponential class of the solutions at the singular points --- the irregular class --- then the category of Stokes data can be described as a category of representations of a quiver with relations. 
For more details, see~\cite[\S 5.1]{Boa21}.

Following the ideas of M.~Kontsevich~\cite[\S 3.4]{Kon09}, we associate with each irregular class a set of curves, called a logarithmic family.
We develop this idea by constructing a correspondence between logarithmic families and irregular classes using a decorated version of logarithmic families, denoted by $\widetilde{J}^\infty$. Restricted to `elementary' objects, this correspondence is  bijective. The precise statement is Theorem~\ref{T1}; see also~\S\ref{SS:elbl}.

The space $J^\infty$, introduced by Kontsevich, is central to the definition of logarithmic families. It can be viewed as a space of truncated jets of curves at infinity. For more context on logarithmic families see~\cite{BelKon07,KS14,KS18,KO21}; in particular, \cite{KO21} gives a coordinate-free definition of $J^\infty$. Here we give an elementary  definition of this space and use it to describe the exponential factor of a formal solution. One can think of $J^\infty$ as a `universal' divisor at infinity of $\mathbb A^2$. In Lemma~\ref{thm:J_transitivity}, we show that the relation defining $J^\infty$ is an equivalence relation. We also prove that the natural action of $\operatorname{Aut}(\mathbb A^2)$ preserves the space of truncated jets of affine plane curves. Both assertions were stated without proof in~\cite[\S 3.4]{Kon09}.

In Appendix~A we formulate and prove analogous statements for power series in place of curves. This reflects the analogy between the truncated jets of affine plane curves and formal Laurent series.

The local data of a $W_{1,\mathbb C}$-module are encoded by the formal meromorphic connection associated with it. We construct this correspondence as the composition of two functors:
\[
\operatorname{Hol}(W_{1,\mathbb C})
\xrightarrow[]{\text{germ}}
\operatorname{Hol}({\mathbb C[[t_c]]\langle\partial\rangle})
\xrightarrow[]{\Phi}
\operatorname{Hol}({\mathbb C((t_c))\langle\partial\rangle}).
\]
Both functors are defined in~\S 2. After applying the functor $\operatorname{germ}$, we obtain a module
that corresponds to an equation over the ring of formal power series. The functor $\Phi$ is a possible extension of the inclusion $\mathbb C[[t_c]]\hookrightarrow\mathbb C((t_c))$ to the category of holonomic $D$-modules.
We also consider a global version of $\Phi$, denoted by $\Phi_{\operatorname{glob}}$ and defined by the family $\{\Phi_c\}_{c\in\mathbb P^1}$. In Theorem~\ref{T2} we show that the action of the group $\mathfrak G \subset \operatorname{Aut}(\mathbb A^2)$ of symplectic
automorphisms on $\widetilde{J}^\infty$ agrees with the action of $\operatorname{Aut}(W_{1,\mathbb C})$ on $\Phi_{\operatorname{glob}}(M)$ for every holonomic $W_{1, \mathbb C}$-module~$M$. The group $\mathfrak G$ is defined in~\S 4.

The Turrittin--Levelt theorem describes the general form of a formal solution near an irregular singular point of a differential equation; see~\cite{Tur55,Levelt}. It can be recast in the language of formal connections. In particular, it gives a complete description of indecomposable connections in terms of elementary blocks $\operatorname{El}(\rho,\varphi,R)$; see~\cite{Sab08}. 
 Elementary blocks describe the formal local type of a meromorphic connection over the field $\mathbb C((t))$.
These results may be viewed as an analogue of the Jordan normal
form for formal meromorphic connections. The third argument $R$ in $\operatorname{El}(\rho,\varphi,R)$ records a regular formal connection.

The local Fourier transform plays a special role in this work: it is a key ingredient in the proofs of Theorems~\ref{T0} and~\ref{T2}. For meromorphic connections, it is defined, roughly speaking, by extending the given formal meromorphic connection to a global object, applying the global Fourier transform, and taking the relevant local part.

More precisely, in~\cite{Sab08, BE04}, given a $\mathbb{C}((t))$-connection \(N\), one chooses a meromorphic connection \(\mathcal N\) whose formalization at the point under consideration is isomorphic to \(N\) and whose singularity at infinity is regular. One then applies the global Fourier transform $\mathcal F$ and takes the formalization of $\mathcal F\mathcal N$ at the appropriate singular point. The resulting object is independent of the choice of \(\mathcal N\)~\cite{BE04}.

For a connection with a singularity at infinity, however, this construction defines the local Fourier transform only under the condition
\[
    N^\nabla
    :=
    \ker\!\left(\nabla:N\longrightarrow N\,\right)
    =0.
\]
The trivial connection shows why:
\(
    N=(\mathbb C((t)), \mathrm{d}),
    \
    N^\nabla=\mathbb C.
\)
As a global extension, we may take
\(
    \mathcal N=
    W_{1,\mathbb C}/W_{1,\mathbb C}\partial_x,
\)
whose Fourier transform is the $\delta$-module
\(
    \mathscr F(\mathcal N)
    \simeq
    W_{1,\mathbb C}/W_{1,\mathbb C}\cdot x.
\)
Passing to the $\mathbb C((t))$-connection annihilates this module:
\[
    \mathbb C((t))
    \otimes_{\mathbb C[[t]]}
    \bigl(W_{1,\mathbb C}/W_{1,\mathbb C}\cdot x\bigr)
    =0.
\]
Thus, in the presence of horizontal sections, the Fourier transform may send part of the connection data to a module supported at a point. This part is lost when we pass to a formal meromorphic connection.

We use D.~Arinkin's approach~\cite{Ari08}, which removes this restriction. The idea is to regard a $\mathbb C((t))$-connection $N$ as a \(W_{1,\mathbb C}\)-module via the inclusion
\[
    W_{1,\mathbb C}
    =\mathbb C[t]\langle\partial_t\rangle
    \hookrightarrow
    \mathbb C((t))\langle\partial_t\rangle.
\]
%Arinkin shows that one can canonically construct from \(N\) an associated \(W_{1,\mathbb C}\)-module to which the global Fourier transform applies, and that the result again determines a formal meromorphic connection. 
Arinkin shows that there is a canonical $W_{1,\mathbb C}$-module associated with 
$N$, to which the global Fourier transform can be applied; the result is again a formal meromorphic connection. 
%which the global Fourier transform can be applied, and that the result again determines a formal meromorphic connection. 
Our choice of the map
\[
\operatorname{Hol}(\mathbb C[[t_c]]\langle\partial\rangle)
\longrightarrow
\operatorname{Hol}(\mathbb C((t_c))\langle\partial\rangle)
\]
is motivated by the compatibility of $\Phi$ with this deﬁnition of the Fourier transform. This approach leads to Theorem~\ref{T0}, in which we use the local stationary phase formula to describe the action of $\operatorname{Aut}(W_{1,\mathbb C})$ on formal meromorphic connections.

In~\S 2 we discuss the formal data carried by a holonomic $W_{1,\mathbb C}$-module and the action of Weyl algebra automorphisms on these data. In~\S 3 we give an elementary definition of $J^\infty$ and introduce its decorated version $\widetilde{J}^\infty$. Its elements are pairs $(j,\mathcal M)$, where $j\in J^\infty$ and $\mathcal M$ is a regular $\mathbb C((t))$-connection. The results of~\S 3 are used in the proofs of Lemmas~\ref{L:cass1} and~\ref{L:cass2}. The main results --- Theorems~\ref{T0}, \ref{T1}, and~\ref{T2} --- are stated and proved in~\S 4.

Appendix~A presents an alternative description of $J^\infty$ in terms of pairs of formal power series. In this description, $J^\infty$ can be interpreted as the divisor at infinity of a compactification $J^\infty_+$ of the affine plane $\mathbb A^2$. The word `compactification' is meant informally here: this is not a compactification in either the algebraic or the topological sense.
In Appendix~B, following the thesis~\cite{Zap15}, we describe regular holonomic $\mathbb C[[t]]\langle\partial\rangle$-modules in terms of quiver representations. This description retains more of the structure of a module germ than its meromorphic localization does. In particular, the data of the quiver representation cannot, in general, be recovered from the elements of $\widetilde{J}^{\infty}$ associated with the module.

\section{Formalization and elementary blocks} \label{s2}
\label{sec:formalization-and-elementary-blocks}

In this section we describe the local data determined by a holonomic $W_{1,\mathbb{C}}$-module near a singular point and the action of automorphisms of $W_{1, \mathbb{C}}$ on these data.

We use the stationary phase formula of~\cite{Sab08} to describe the Fourier transform of elementary blocks. For the definition of the formalization functors $\Psi$ and $\Phi$ and their interaction with the Fourier transform, we follow~\cite{Ari08}.

\subsection{Formal germs of a module} \label{s2.1}
\label{subsec:formalization}

Let
\[
W_{1,\mathbb C}
=
\mathbb C[x]\langle\partial_x\rangle,
\qquad
[\partial_x,x]=1,
\]
and let $\mathcal M$ be a left holonomic $W_{1,\mathbb C}$-module. Recall that, for the first Weyl algebra, the holonomic modules are precisely the cyclic modules that are not free~\cite{Stafford78}.

For a local coordinate $t$, set
\[
\mathscr D_t
=
\mathbb C[[t]]\langle\partial_t\rangle,
\qquad
\mathscr D_t^\star
=
\mathbb C((t))\langle\partial_t\rangle.
\]
\begin{definition}
A \emph{formal meromorphic connection} is a finite-dimensional $\mathbb C((t))$-vector space $V$ equipped with a connection $\nabla$. Here a connection means a $\mathbb C$-linear operator on $V$ satisfying the Leibniz rule $\nabla(fv)=f'v+f\nabla(v)$. Once a basis of $V$ has been chosen, the images of the basis vectors under $\nabla$ determine its connection matrix.
\end{definition}
\begin{definition}
(1) A formal meromorphic connection is \emph{regular} if there is a basis of $V$ in which every entry of its connection matrix has at most a simple pole.

(2) A formal meromorphic connection is \emph{irregular} if it is not regular. If, moreover, $V$ has no nonzero $\nabla$-stable subspace $V'$ for which $(V', \nabla)$ is regular, we call the connection \emph{purely irregular}.

(3) For every formal connection $(V, \nabla)$, there is a canonical decomposition of $V$ into two $\nabla$-stable subspaces, $V_{\operatorname{reg}}\oplus V_{\operatorname{irr}}$, such that $(V_{\operatorname{reg}}, \nabla)$ is regular and $(V_{\operatorname{irr}}, \nabla)$ is purely irregular; see~\cite[\S 3.3]{Ari08}.
\end{definition}

Every formal meromorphic connection $(V, \nabla)$ naturally defines a $\mathscr D_t^\star$-module in which $\partial$ acts as $\nabla$. We denote this module by $\mathscr D_t^\star(V, \nabla)$. This construction gives an equivalence of the corresponding categories. We say that $\mathscr D_t^\star(V, \nabla)$ corresponds to the formal connection $(V, \nabla)$.

For $c\in\mathbb C$, choose the local coordinate $t_c=x-c$. Using the homomorphism
\[
W_{1,\mathbb C}\longrightarrow\mathscr D_{t_c},
\qquad
x\longmapsto c+t_c,
\qquad
\partial_x\longmapsto\partial_{t_c},
\]
we make the following definition.
\begin{definition}
The \emph{germ of $\mathcal M$ at a finite point $c$} is the holonomic $\mathscr D_{t_c}$-module
\[
\mathcal M_c
:=
\mathscr D_{t_c}
\otimes_{W_{1,\mathbb C}}\mathcal M.
\]
\end{definition}

The category of holonomic $\mathscr D_{t_c}$-modules is the direct sum of its regular and irregular subcategories~\cite[Remark~3.6]{Sab08}. Thus every holonomic module can be written as
\[
\mathcal M_c
\simeq
\mathcal M_{c,\operatorname{reg}}
\oplus
\mathcal M_{c,\operatorname{irr}},
\]
where $\mathcal M_{c,\operatorname{reg}}$ has meromorphic localization
$$\mathcal M_{c,\operatorname{reg}}^\star:=\mathbb C((t_c))\otimes_{\mathbb C[[t_c]]}\mathcal M_{c,\operatorname{reg}}$$ corresponding to a regular
formal connection. The similarly defined $\mathcal M^\star_{c,\operatorname{irr}}$ corresponds to a purely irregular formal connection.
The natural map $\mathcal M_{c,\operatorname{irr}}\to \mathcal M^\star_{c,\operatorname{irr}}$ is an isomorphism, whereas $\mathcal M_{c,\operatorname{reg}}\to \mathcal M^\star_{c,\operatorname{reg}}$ may have a kernel, which contains a direct sum of copies of the $\delta$-module.

To define the germ at infinity, set $t_\infty=x^{-1}$. The formulas
\[
x=t_\infty^{-1},
\qquad
\partial_x=-t_\infty^2\partial_{t_\infty}
\]
define a homomorphism
\[
W_{1,\mathbb C}
\longrightarrow
\mathscr D_{t_\infty}^\star.
\]
\begin{definition}
The \emph{germ of $\mathcal M$ at infinity} is the formal meromorphic connection
\[
\mathcal M_\infty^{\star}
:=
\mathscr D_{t_\infty}^\star
\otimes_{W_{1,\mathbb C}}\mathcal M.
\]
\end{definition}

Unlike a germ at a finite point, every $\mathcal M^{\star}_\infty$ is a meromorphic connection. Its regular and irregular parts are also $\mathbb C((t_\infty))$-connections.

\subsection{Formalization functors}\label{s2.2}
\label{subsec:elementary-connections}
The local Fourier transform is defined for formal meromorphic connections. To apply it to the germ of a holonomic module at a finite point, we must first associate a formal meromorphic connection with that germ. We call the functors used for this purpose \emph{formalization functors}.

The simplest example of a formalization functor is meromorphic localization:
\[
\Psi_c:
\operatorname{Hol}\big({\mathscr{D}_{t_c}}\big)
\longrightarrow
\operatorname{Hol}\big(\mathscr{D}_{t_c}^\star\big),
\qquad
\Psi_c(N_c)
=
\mathbb C ((t_c))\otimes_{\mathbb C[[t_c]]} N_c.
\]

However, $\Psi$ vanishes on modules supported at points. For instance, for \[\delta_0=W_{1,\mathbb C}/W_{1,\mathbb C}\cdot x, \] we have \(\Psi\bigl((\delta_0)_0\bigr)=0.\)
On the other hand,
$$\mathcal F(\delta_0)\simeq W_{1,\mathbb C}/W_{1,\mathbb C}\cdot\partial_x,$$
and the meromorphic localization of its germ at \(0\) is nonzero.
Thus \(\Psi\) alone does not capture the behavior of local data under the Fourier transform. We need another formalization functor that retains the information carried by \(\delta\)-modules.

Write $\iota$ for the inclusion
$\mathbb C[[t_c]]\hookrightarrow\mathbb C((t_c))$. For $M\in\operatorname{Hol}\big(\mathcal D_{t_c}^\star\big)$, let $\iota_*M$ be the same module regarded as a $\mathcal D_{t_c}$-module.
The functor $\iota_!$, dual to $\iota_*$, is defined by
\[
\iota_!N_c
:=
\mathbb D_{t_c}\bigl(\iota_*(\mathbb D_{t_c}^\star N_c)\bigr),
\]
where $\mathbb D_{t}^\star$ and $\mathbb D_{t}$ are the Verdier duality functors.

The formal vanishing cycles functor
\[
\Phi_c:
\operatorname{Hol}(D_{t_c})
\longrightarrow
\operatorname{Hol}(D_{t_c}^\star)
\]
is defined as the left adjoint of $\iota_!$. It is also the right adjoint of $\iota_*$. Together these functors form a chain of adjunctions:
\begin{equation}
\Psi_c
\dashv
\iota_*
\dashv
\Phi_c
\dashv
\iota_!
\dashv
\Psi_c.
\label{eq:adjunction-chain}
\end{equation}

If the original object is already a formal meromorphic connection, both formalization functors preserve it up to natural isomorphism:
\[
\Psi_c(\iota_*N_c)\simeq N_c,
\qquad
\Phi_c(\iota_*N_c)\simeq N_c,
\]
where $N_c=\big(\mathbb C((t_c))^r, \nabla\big)$.
In particular,
\[
\Psi_c(\mathcal M_{c,\operatorname{irr}})
\simeq
\Phi_c(\mathcal M_{c,\operatorname{irr}})
\simeq
\mathcal M_{c,\operatorname{irr}}.
\]
For uniformity, it is convenient to define $\Phi_\infty$ as the identity functor on $\operatorname{Hol}(\mathcal D_{t_\infty}^\star)$.

The functor $\Phi$ has an explicit description in terms of microlocalization.
Set $$\mu D=\big\{\sum_{i=-\infty}^k a_i(t)\partial^i\big|\ a_i(t)\in\mathbb{C}[[t]],\ k\in\mathbb Z\big\};$$ then $\mu D \otimes_{\mathscr{D}_t}M$ is the microlocalization of $M$ at the origin. \begin{remark}
Using the Fourier transform, it is easy to check that $\mu D$ is indeed a ring.
\end{remark} We have the following formula; see~\cite[Example~3.2 and Theorem~3.3]{Ari08}:
$$\mu D \otimes_{\mathscr{D}_{t_c}}M=\iota_!\Phi_c(M),\qquad M\in\operatorname{Hol}(\mathscr{D}_{t_c}).$$
In particular, the isomorphism remains valid when both sides are viewed as $\mathscr{D}_{t_c}$-modules; we denote this restriction by $\bullet_{|\mathbb{C}[[t_c]]}$. Applying $\Psi_c$ to both sides and using the adjunction between $\iota_!$ and $\Psi$, we obtain
$$\Psi_c\big((\mu D \otimes_{\mathscr{D}_{t_c}}M_c)_{|\mathbb{C}[[t_c]]}\big)=\Psi_c\big( \iota_!\Phi_c(M_c)\big)=\Phi_c(M_c).$$
Now the left-hand side is simply the restriction to $\mathscr{D}_t^\star$, so
$$(\mu D \otimes_{\mathscr{D}_{t_c}}M_c)_{|\mathbb{C}((t_c))}=\Phi(M_c).$$
\begin{remark}
The preceding formula may be taken as the definition of $\Phi$.
\end{remark}
Set $\Sigma_0=\{\lambda\in\mathbb C\mid
0\leq\operatorname{Re}\lambda<1\text{ and }\hspace{2pt}\lambda\ne0\}$. The microlocal description allows us to compute $\Phi$ on indecomposable regular $\mathcal D_t$-modules:
\begin{center}
\begin{tabular}{c|c}
Indecomposable module $N_c$ & $\Phi_c(N_c)$ \\
\hline
\(\mathscr D_{t_c}/\mathscr D_{t_c}({t_c}\partial_{t_c}-\lambda)^n, \lambda\in\Sigma_0\)
&
\(\mathscr D_{t_c}^\star/\mathscr D_{t_c}^\star({t_c}\partial_{t_c}-\lambda)^n
\simeq
\mathscr D_{t_c}^\star/\mathscr D_{t_c}^\star({t_c}\partial_{t_c}-(\lambda-1))^n\)
\\ \hline
\(\mathscr D_{t_c}/\mathscr D_{t_c}({t_c}\partial_{t_c})^n\)
&
\(\mathscr D_{t_c}^\star/\mathscr D_{t_c}^\star({t_c}\partial_{t_c})^n\)
\\ \hline
\(\mathscr D_{t_c}/\mathscr D_{t_c}(\partial_{t_c} {t_c})^n\)
&
\(\mathscr D_{t_c}^\star/\mathscr D_{t_c}^\star({t_c}\partial_{t_c})^n\)
\\ \hline
\(\mathscr D_{t_c}/\mathscr D_{t_c}(\partial_{t_c} {t_c})^m\partial_{t_c}\)
&
\(\mathscr D_{t_c}^\star/\mathscr D_{t_c}^\star({t_c}\partial_{t_c})^m\)
\\ \hline
\(\mathscr D_{t_c}/\mathscr D_{t_c}({t_c}\partial_{t_c})^m {t_c}\)
&
\(\mathscr D_{t_c}^\star/\mathscr D_{t_c}^\star({t_c}\partial_{t_c})^{m+1}\)
\end{tabular}
\end{center}
In particular, $\Phi_c$ does not vanish on $\delta$-modules.

To simplify notation, we henceforth omit the subscripts on $\Phi$ and $\Psi$: the point at which the germ is taken will always be clear from the context.

\subsection{Elementary blocks and \texorpdfstring{$\mathbb C ((t))$}{C((t))}-connections}
After applying a formalization functor to a module germ, we can work with $\mathbb{C} ((t))$-connections. We now recall their description in terms of elementary blocks~\cite{Sab08}.

Consider the map
\[
\rho:
\mathbb C[[t]]
\longrightarrow
\mathbb C[[u]],
\qquad
t\longmapsto\rho(u),
\]
where
\(
\rho(u)\in u^p\mathbb C[[u]]^\times,\
p:=\operatorname{ord}_u\rho\geq1.
\)
After a formal change of parameter, we may assume that $\rho(u)=u^p$.

For a $\mathbb C((u))$-connection $N$, its direct image $\rho_+N$ is defined on the same complex vector space by
\[
f(t)m=f(\rho(u))m,
\qquad
\nabla_{\partial_t}m
=
\frac1{\rho'(u)}\nabla_{\partial_u}m.
\]

For $\varphi\in\mathbb C((u))$, let $E^\varphi$ denote the rank-one exponential connection
\[
\nabla_{\partial_u}(f e_\varphi)
=
\bigl(\partial_u f+\varphi'(u)f\bigr)e_\varphi.
\] With this convention, $E^\varphi$ and $E^\psi$ are isomorphic whenever $\varphi-\psi\in\mathbb C[[u]]$; see~\cite{Sab08}.

\begin{definition}
Let $\rho$ be as above, let $\varphi\in\mathbb C((u))$, and let $R$ be a regular $\mathbb C((u))$-connection. The associated \emph{elementary formal connection}, or \emph{elementary block}, is
\[
\operatorname{El}(\rho,\varphi,R)
:=
\rho_+\bigl(E^\varphi\otimes R\bigr).
\]
\end{definition}

A presentation of an elementary block in terms of $\rho$, $\varphi$, and $R$ is called \emph{minimal} if the connection admits no such presentation with ramification of smaller degree.

There is considerable freedom in the presentation of an elementary block. Two minimal presentations $\operatorname{El}(\rho,\varphi,R)$ and $\operatorname{El}(\widetilde{\rho}, \widetilde{\varphi}, \widetilde{R})$ define isomorphic blocks if and only if there is an $h\in u\mathbb{C}[[u]]$ with $h'(0)\ne0$ such that $R\simeq \widetilde{R}\circ h$ and \[\rho=\widetilde{\rho}(h),\quad\varphi-\widetilde{\varphi}(h)\in \mathbb{C}[[u]].\]

If $\varphi\notin\mathbb C[[u]]$ and
\[
p=\operatorname{ord}\rho,
\qquad
q=-\operatorname{ord}\varphi,
\qquad
r=\operatorname{rk}R,
\]
then
\begin{equation}
\operatorname{rk}
\operatorname{El}(\rho,\varphi,R)
=
pr,
\qquad
\operatorname{slope}
\operatorname{El}(\rho,\varphi,R)
=
\frac qp.
\label{eq:slope-elementary-block}
\end{equation}
The term `slope' comes from the Newton polygon procedure for computing formal solutions.
A regular block has a minimal presentation of the form $\operatorname{El}(u, 0, R)$. A formal meromorphic connection is regular if and only if its slope is zero.

The following result is well known; see~\cite[Corollary~3.3]{Sab08} and~\cite{Tur55,Levelt,BabbittVaradarajan,vdPutSinger}.
\begin{theorem}[Refined Turrittin--Levelt decomposition] \label{th:levtur}
Every finite-dimensional formal meromorphic connection $N$ admits a decomposition
\begin{equation}
N
\simeq
\bigoplus_{i\in I}
\operatorname{El}(\rho_i,\varphi_i,R_i).
\label{eq:refined-turrittin-levelt}
\end{equation}
The presentations can be chosen to be minimal, with the connections $(\rho_i)_+E^{\varphi_i}$ irreducible and pairwise nonisomorphic. Under these conditions, the decomposition is unique up to permutation and isomorphism of the elementary blocks.
\end{theorem}
Uniqueness follows from the Krull--Schmidt theorem for the category $\operatorname{Hol}(\mathcal D_t^\star)$, whose objects all have finite length; see~\cite{Atiyah}.

After applying $\Phi$, the regular part of a module germ can also be written as an elementary block:
\(
\operatorname{El}(u, 0, \Phi(M_{c, \operatorname{reg}})).
\)

For every point \(c\in\mathbb P^1\), the Turrittin--Levelt decomposition gives a decomposition of the category of formal meromorphic connections by slope:
\[
    \operatorname{Hol}\bigl(\mathcal D_{t_c}^{\star}\bigr)
    =
    \bigoplus_{q\in\mathbb Q_{\geqslant 0}}
    \operatorname{Hol}\bigl(\mathcal D_{t_c}^{\star}\bigr)^q,
\]
where the superscript \(q\) indicates that every elementary block of the connection has slope~\(q\).

For an elementary block
\(E=\operatorname{El}(\rho,\varphi,R)_\infty\)
of slope at most one, set
\[
    \alpha(E)=
    \begin{cases}
        0, & \operatorname{slope}E<1,\\[1mm]
        \displaystyle
        -\left.\rho^2\frac{\varphi'}{\rho'}\right|_{u=0},
        & \operatorname{slope}E=1.
    \end{cases}
\]
It is easy to see that this number does not depend on the choice of a minimal presentation.
For \(\alpha_0\in\mathbb C\), write
\(\operatorname{Hol}(\mathcal D_{t_\infty}^{\star})^{\leqslant1,(\alpha_0)}\)
for the full subcategory of connections whose elementary blocks \(E\) all have slope at most one and satisfy \(\alpha(E)=\alpha_0\). In particular,
\[
    \operatorname{Hol}(\mathcal D_{t_\infty}^{\star})^{\leqslant1,(0)}
    =
    \operatorname{Hol}(\mathcal D_{t_\infty}^{\star})^{<1}.
\]

By~\cite{Ari08}, we have the decomposition
\[
    \operatorname{Hol}(\mathcal D_{t_\infty}^{\star})
    =
    \operatorname{Hol}(\mathcal D_{t_\infty}^{\star})^{>1}
    \oplus
    \bigoplus_{\alpha_0\in\mathbb C}
    \operatorname{Hol}(\mathcal D_{t_\infty}^{\star})^{\leqslant1,(\alpha_0)}.
\]
For a connection \(N\), denote the corresponding components by \(N^{>1}\) and \(N^{\leqslant1,(\alpha_0)}\).

This decomposition is needed to define the local Fourier transform at infinity: the transform sends \(N^{\leqslant1,(\alpha_0)}\) to a formal meromorphic connection at \(\alpha_0\); see~\S\ref{s2.3}.

\begin{example}
Consider the Airy module
\[
    \mathcal{A}i
    =
    W_{1,\mathbb C}/W_{1,\mathbb C}(\partial_x^2-x).
\]
At infinity, set \(t=x^{-1}=u^2\). Substituting a formal solution into \(y''-xy=0\) gives
\[
    y_\pm
    =
    \exp\left(\pm\frac{2}{3u^3}\right)
    u^{1/2}a_\pm(u),
    \qquad
    a_\pm(u)\in1+u^3\mathbb C[[u^3]].
\]
Hence
\[
    \mathcal{A}i_\infty
    \simeq
    \operatorname{El}\left(
        u^2,
        -\frac{2}{3u^3},
        R
    \right),
    \qquad
    R=
    \left(
        \mathbb C((u)),
        \mathrm d+\frac12\frac{\mathrm du}{u}+\frac{\mathrm d a_-}{a_-}
    \right).
\]
Since \(a_-\in\mathbb C[[u]]^\times\), the change of basis
\(e\mapsto a_-(u)^{-1}e\) removes the last term; as a result $R\simeq\left(
        \mathbb C((u)),
        \mathrm d+\frac12\frac{\mathrm du}{u}
    \right)$.
Thus the elementary block has rank \(2\) and slope \(3/2\).
\end{example}

\subsection{Automorphisms of the first Weyl algebra}\label{s2.3}
\label{subsec:weyl-automorphisms}

We now describe how Weyl algebra automorphisms act on the constructions above. An automorphism acts on holonomic \(W_{1,\mathbb C}\)-modules, and hence on their local germs: the germs of a module are sent to the germs of the transformed module. After formalization, each germ decomposes into elementary blocks, giving an induced action on collections of elementary blocks. The group $\operatorname{Aut}(W_{1,\mathbb C})$ is generated by the following automorphisms~\cite{Dix68}:
\begin{equation}
\begin{aligned}
\sigma_P&:\quad
x\longmapsto x,
&
\partial_x&\longmapsto\partial_x+P(x),
&&
P\in\mathbb C[x],
\\
\xi_\alpha&:\quad
x\longmapsto\alpha x,
&
\partial_x&\longmapsto\alpha^{-1}\partial_x,
&&
\alpha\in\mathbb C^\times,
\\
\mathcal F&:\quad
x\longmapsto\partial_x,
&
\partial_x&\longmapsto-x.
\end{aligned}
\label{eq:weyl-generators}
\end{equation}

We use the definition of the local Fourier transform from~\cite[Theorem~A]{Ari08}. Let \(c\in\mathbb C\), and let \(N\) be a formal meromorphic connection at \(c\). Write \(\iota_c\) for the inclusion \(\mathbb C[[t_c]]\hookrightarrow\mathbb C((t_c))\), and \(j_{c;*}\) for the restriction of scalars along the homomorphism
\[
W_{1,\mathbb C}\to\mathbb C[[t_c]]\langle\partial_{t_c}\rangle,
\quad
x\mapsto c+t_c,\quad \partial\mapsto\partial_{t_c}.
\]
We define \(\mathscr F^{(c,\infty)}(N)\) to be the formal meromorphic connection at infinity characterized by the natural isomorphism
\[
    \mathcal F\bigl(j_{c;*}\iota_{c;!}N\bigr)
    \simeq
    \dot{\jmath}_{\infty; *}
    \bigl(\mathscr F^{(c,\infty)}(N)\bigr).
\]
Here \(\mathcal F\) is the Fourier transform of \(W_{1,\mathbb C}\)-modules defined by
\(x\mapsto\partial,\ \partial\mapsto-x\).

For a formal meromorphic connection \(N\) near infinity, with
\(N\in\operatorname{Hol}(\mathscr D_{t_\infty}^\star)^{\leqslant1,(\alpha)}\),
we define \(\mathscr F^{(\infty,\alpha)}(N)\) as the formal meromorphic connection at \(\alpha\) characterized by the natural isomorphism
\begin{equation}
    \mathcal F\bigl(\dot{\jmath}_{\infty; *}N\bigr)
    \simeq
    j_{\alpha;*}\iota_{\alpha;!}
    \bigl(\mathscr F^{(\infty,\alpha)}(N)\bigr).
    \label{eq:phaseC}
\end{equation}
For
\(N\in\operatorname{Hol}(\mathscr D_{t_\infty}^\star)^{>1}\)
we define
\(\mathscr F^{(\infty,\infty)}(N)
\in\operatorname{Hol}(\mathscr D_{t_\infty}^\star)^{>1}\)
by the natural isomorphism
\begin{equation}
    \mathcal F\bigl(\dot{\jmath}_{\infty; *}N\bigr)
    \simeq
    \dot{\jmath}_{\infty; *}
    \bigl(\mathscr F^{(\infty,\infty)}(N)\bigr),
    \label{eq:phaseInfty}
\end{equation}
where \(\dot{\jmath}_{\infty; *}\) is the restriction of scalars along
\[
W_{1,\mathbb C}\to\mathscr D_{t_\infty}^\star,
\quad
(x,\partial)\mapsto
(t_\infty^{-1},-t_\infty^2\partial_{t_\infty})
.\]

The existence of these functors and their uniqueness up to natural isomorphism follow from~\cite[Theorem~A and Lemma~4.3]{Ari08}. By~\cite[\S 2.2]{Ari08}, the local Fourier transform defined above agrees with that of~\cite{BE04}.

\begin{theorem}
For a holonomic \(W_{1,\mathbb C}\)-module \(\mathcal M\), the following stationary phase formulas hold:
\begin{align}
\label{eq:Phi-Fourier}
&\Phi\bigl((\mathcal F\mathcal M)_\infty\bigr)
=
\mathscr F^{(\infty,\infty)}
\bigl(\Phi(\mathcal M_\infty)^{>1}\bigr)
\oplus
\bigoplus_{c\in\mathbb A^1}
\mathscr F^{(c,\infty)}\Phi(\mathcal M_c),\\
&\Phi\bigl((\mathcal F\mathcal M)_\alpha\bigr)
=
\mathscr F^{(\infty,\alpha)}
\left(
\Phi(\mathcal M_\infty)^{\leqslant1,(\alpha)}
\right).
\end{align}
The inequalities in the superscripts specify the components with the indicated slopes.
\end{theorem}
\begin{proof}
The assertion follows from~\cite[Corollary~2.3]{Ari08} and the equality $\Phi (\mathcal{ M_\infty})=\Psi( \mathcal{ M_\infty})$.
\end{proof}

We now give explicit formulas. 
For $m\in\mathbb Z$, define the regular connection of rank one
\begin{equation} \label{eq:Lm}
L_m
=
\left(
\mathbb C((u)),
\mathrm d-\frac m2\frac{\mathrm du}{u}
\right).
\end{equation}

The local Fourier transforms of elementary blocks are then given as follows.

$\bullet\ \mathscr F^{(c,\infty)}:$ let
\(
E
=
\operatorname{El}(\rho,\varphi,R)_c,
\
p=\operatorname{ord}\rho,
\
q=-\operatorname{ord}\varphi>0.
\)
Then
$$\mathscr F^{(c,\infty)}E
\simeq
\operatorname{El}
(\widehat\rho,\widehat\varphi,\widehat R)_\infty,
$$
where
\begin{equation}
\begin{aligned}
\widehat\rho
&=
\frac{\rho'}{\varphi'},
\\
\widehat\varphi
&=
\varphi
-
\frac{\rho}{\rho'}\varphi'
-
\frac{c}{\widehat\rho},
\\
\widehat R
&=
R\otimes L_q.
\end{aligned}
\label{eq:fourier-finite-irregular-formulas-final}
\end{equation}
Here $\operatorname{ord}\widehat\rho=p+q$. For $c=0$, the exponential factor
$\widehat\varphi$ has a pole of order $q$, so
\[
\operatorname{slope}\mathscr F^{(c,\infty)} E
=
\frac{q}{p+q}
<1.
\]
For $c\ne0$, the term $-c/\widehat\rho$ has a pole of order $p+q$; hence
$\widehat E$ has slope one.

$\bullet$ For a regular block
\(
E=\operatorname{El}(u ,0,R)_c
\), we have
\begin{equation}
   \mathscr F^{(c,\infty)}E
\simeq
\operatorname{El}
\left(u, -\frac{c}{u}, R\right)_\infty.
\end{equation}
For $c=0$ the block remains regular; for $c\ne0$ its image has slope one.

$\bullet \ \mathscr F^{(\infty,\infty)}:$ for a block
\(
E
=
\operatorname{El}(\rho,\varphi,R)_\infty,
\
q>p,
\)
the local transform $\mathscr F^{(\infty,\infty)}$ is given by
\begin{equation}
\begin{aligned}
\widehat\rho
&=
-\frac{\rho'}{\varphi'\rho^2},
\\
\widehat\varphi
&=
\varphi+\frac{\rho}{\rho'}\varphi',
\\
\widehat R
&=
R\otimes L_q,
\end{aligned}
\label{eq:fourier-infinity-greater-one-final}
\end{equation}
and
\[
\operatorname{ord}\widehat\rho=q-p,
\qquad
\operatorname{slope}\mathscr F^{(\infty,\infty)} E
=
\frac{q}{q-p}>1.
\]

$\bullet \ \mathscr F^{(\infty,0)}:$ if
\(
E
=
\operatorname{El}(\rho,\varphi,R)_\infty,
\
0<q<p,
\)
then $E$ is sent to the point $0$, and
\(
\mathscr F^{(\infty,0)}E
\simeq
\operatorname{El}
(\widehat\rho,\widehat\varphi,\widehat R)_0,
\)
where
\begin{equation}
\begin{aligned}
\widehat\rho
&=
-\rho^2\frac{\varphi'}{\rho'},
\\
\widehat\varphi
&=
\varphi+\frac{\rho}{\rho'}\varphi',
\\
\widehat R
&=
R\otimes L_q.
\end{aligned}
\label{eq:fourier-infinity-less-one-final}
\end{equation}
Here
\[
\operatorname{ord}\widehat\rho=p-q,
\qquad
\operatorname{slope}\mathscr F^{(\infty,0)} E
=
\frac{q}{p-q}.
\]

$\bullet \ \mathscr F^{(\infty,\widehat c)}:$ finally, for
\(\operatorname{El}(\rho,\varphi,R)_\infty\) with \(q=p\), set
\begin{equation}
\widehat c
=
\left.
\left(
-\rho^2\frac{\varphi'}{\rho'}
\right)
\right|_{u=0},
\label{eq:fourier-slope-one-point}
\end{equation}
and
\begin{equation}
\widehat\rho
=
-\rho^2\frac{\varphi'}{\rho'}
-\widehat c,
\qquad
\widehat\varphi
=
\varphi+\frac{\rho}{\rho'}\varphi'.
\label{eq:fourier-slope-one-data}
\end{equation}
The target point is the finite point \(\widehat c\). If
\(\widehat\varphi\notin\mathbb C[[u]]\), set
\(\widehat q=-\operatorname{ord}\widehat\varphi\). Then
\begin{equation}
\mathscr F^{(\infty,\widehat c)}
\operatorname{El}(\rho,\varphi,R)_\infty
\simeq
\operatorname{El}
\left(
\widehat\rho,
\widehat\varphi,
R\otimes L_{\widehat q}
\right)_{\widehat c}.
\label{eq:fourier-slope-one-irregular}
\end{equation}

If \(\widehat\varphi\in\mathbb C[[u]]\), minimality implies
\(p=1\) and \(\varphi=\widehat c/u\). In this case,
\begin{equation}
\mathscr F^{(\infty,\widehat c)}
\operatorname{El}(u,\widehat c/u,R)_\infty
\simeq
\operatorname{El}(u, 0, R)_{\widehat c}.
\label{eq:fourier-slope-one-regular}
\end{equation}

$\bullet$ A regular connection at infinity is sent to a regular connection at \(0\):
\begin{equation}
\mathscr F^{(\infty,0)}
R_\infty
\simeq
R_0,
\label{eq:fourier-regular-infinity}
\end{equation}
where \(R_0\simeq R_\infty\) after the natural renaming of the formal parameter.

The other generators act in a straightforward way. For the triangular automorphism
\[
\sigma_P(x)=x,
\qquad
\sigma_P(\partial_x)=\partial_x+P(x),
\qquad
P(x)=\mathcal P'(x),
\]
the transformed module is
\(
\sigma_P\mathcal M
\simeq
\mathcal M\otimes E^{\mathcal P(x)}.
\)
At a finite point $c$, the series $\mathcal P(c+t_c)$ lies in $\mathbb C[[t_c]]$, so the exponential
connection $E^{\mathcal P(c+t_c)}$ is trivial. Hence
\begin{equation}
\Psi(\sigma_P\mathcal M_c)
\simeq
\Psi(\mathcal M_c),
\qquad
\Phi\bigl((\sigma_P\mathcal M)_c\bigr)
\simeq
\Phi(\mathcal M_c).
\label{eq:sigma-Psi-Phi}
\end{equation}
On elementary blocks, the action is
\begin{equation}
\operatorname{El}(\rho,\varphi,R)_c
\longmapsto
\operatorname{El}
\bigl(
\rho,
\varphi+\mathcal{P} (c+\rho),
R
\bigr)_c
\simeq
\operatorname{El}(\rho,\varphi,R)_c.
\label{eq:sigma-finite-final}
\end{equation}

The situation at infinity is similar. If \(\mathcal P\) is nonconstant, then \(\mathcal P(t_\infty^{-1})\) has a nonzero polar part. Since the germ at infinity is already a formal meromorphic connection over \(\mathbb C((t_\infty))\), an analogue of~\eqref{eq:sigma-Psi-Phi} holds. Writing the connection as a sum of elementary blocks, we obtain
\begin{equation}
\operatorname{El}(\rho,\varphi,R)_\infty
\longmapsto
\operatorname{El}
\left(
\rho,
\varphi+\mathcal P\left(\frac1\rho\right),
R
\right)_\infty.
\label{eq:sigma-infinity-final}
\end{equation}
\begin{remark}\label{rem:trans}
The formula above shows that an irregular elementary block can become regular under $\sigma_P$ only if its minimal presentation is unramified.
\end{remark}

The scaling
\[
\xi_\alpha(x)=\alpha x,
\qquad
\xi_\alpha(\partial_x)=\alpha^{-1}\partial_x,
\qquad
\alpha\in\mathbb C^\times,
\]
is the change of coordinate $x_\alpha=\alpha x$. The point $c$ is sent to
$\alpha c$, and
\(
t_{\alpha c}
=
x_\alpha-\alpha c
=
\alpha t_c.
\)
After identifying the fields by
$t_{\alpha c}=\alpha t_c$, we obtain
\begin{equation}
\Psi\bigl((\xi_\alpha\mathcal M)_{\alpha c}\bigr)
\simeq
\Psi(\mathcal M_c),
\qquad
\Phi\bigl((\xi_\alpha\mathcal M)_{\alpha c}\bigr)
\simeq
\Phi(\mathcal M_c),
\label{eq:scaling-Psi-Phi}
\end{equation}
and
\begin{equation}
\operatorname{El}(\rho,\varphi,R)_c
\longmapsto
\operatorname{El}
(\alpha\rho,\varphi,R)_{\alpha c}.
\label{eq:scaling-finite-final}
\end{equation}

At infinity, $t_\infty\mapsto \alpha^{-1}t_\infty$, so an analogue of~\eqref{eq:scaling-Psi-Phi} holds, and
\begin{equation}
\operatorname{El}(\rho,\varphi,R)_\infty
\longmapsto
\operatorname{El}
(\alpha^{-1}\rho,\varphi,R)_\infty.
\label{eq:scaling-infinity-final}
\end{equation}

\section{The decorated space \texorpdfstring{$\widetilde{J}^\infty$}{J-tilde infinity}}
\label{sec:J-infinity}
In this section we define $\widetilde{J}^\infty$, the space of pairs consisting of a truncated jet of an affine plane curve at infinity and a regular \(\mathbb C ((t))\)-connection.
This space will give a complete description of \(\Phi(\mathcal M_c)\) in~\S 4. We also define an action of automorphisms of \(\mathbb A ^2\) on $\widetilde{J}^\infty$. Its compatibility with the action of \(\operatorname{Aut} W_{1, \mathbb C}\) will be proved in the next section.

\subsection{The space of truncated jets of plane curves}
The space of truncated jets of affine plane curves $J^\infty$ was introduced by M.~Kontsevich in~\cite{Kon09}, in the context of $p$-determinants of $D$-modules in one variable.
\begin{Cit}[{\cite[{\S 3.4}]{Kon09}}]
Let $\mathbf{k}$ be an algebraically closed field of characteristic zero. The set of planar curves $Curves_{\mathbb{A}_{\mathbf{k}}^{2}}$ understood as effective divisors in $\mathbb{A}_{\mathbf{k}}^{2}$, is the same as the quotient of the set of non-zero polynomials $P \in \mathbf{k}[x, y] \setminus \{0\}$ modulo multiplicative constant. Hence it carries a natural structure of an ind-scheme, it is an infinite-dimensional projective space
\[Curves_{\mathbb{A}_{\mathbf{k}}^{2}} = \mathbb{P}^{\infty}(\mathbf{k}) = \mathbb{P}(\mathbf{k}[x,y]) = \lim_{\rightarrow}\mathbb{P}^{\frac{(p + 1)(p + 2)}{2} -1}(\mathbf{k}).\]
The group $\operatorname {Aut} \mathbb{A}_{\mathbf{k}}^{2}$ acts by automorphisms of this ind-scheme.

Our goal here is to introduce an equivalence relation on $Curves_{\mathbb{A}_{\mathbf{k}}^{2}}$ invariant under $\operatorname {Aut} \mathbb{A}_{\mathbf{k}}^{2}$ such that all equivalence classes will be sets of $\mathbf{k}$-points of constructible sets, and for any cyclic $\mathcal{D}(\mathbb{A}_{\mathbf{k}}^{1})$-module $M$ its arithmetic support will belong to one such an equivalence class. These equivalence classes we will call logarithmic families because they have a characterization in terms logarithmic divergence of certain integrals.

First, we introduce certain set $J_{\mathbf{k}}^{\infty}$ associated with $\mathbb{A}_{\mathbf{k}}^{2}$. It can be thought as truncated jets of algebraic curves in $\mathbb{P}_{\mathbf{k}}^{2}$ intersecting the projective line at infinity $\mathbb{P}_{\mathbf{k}}^{1} = \mathbb{P}_{\mathbf{k}}^{2} \setminus \mathbb{A}_{\mathbf{k}}^{2}$.
\end{Cit}

We give a slightly different definition of \(J^\infty\) and, following Kontsevich's ideas, show that the space of truncated jets of affine plane curves is well suited to describing elementary blocks.

\begin{definition}
A \emph{jet of an affine plane curve at infinity} is a triple $(C,\phi,p)$, where $C$ is a smooth complete curve, $\phi\colon C\dashrightarrow\mathbb A^2$ is a rational map birational onto its image, and $p\in C$ is a point at which $\phi$ is not defined as a map to $\mathbb A^2$ and which is sent to $\mathbb P^2\setminus\mathbb A^2$ by its extension. We denote this extension by
\(
\widehat\phi\colon C\longrightarrow\mathbb P^2
\).
\end{definition}

For a smooth point $p$ of $C$ and any local parameter $s$, we have
\[
\widehat{\mathscr O}_{C,p}\simeq\mathbb C[[s]],
\qquad
\operatorname{Frac}\widehat{\mathscr O}_{C,p}\simeq\mathbb C((s)).
\]
Thus the coordinate functions of such a jet may be regarded as Laurent series in $s$. We write $f(s)=O(s^N)$ to mean $f\in s^N\mathbb C[[s]]$.
\begin{definition}
\label{J_def}
Let $(C_I,\phi_I,p_I)$ and $(C_{II},\phi_{II},p_{II})$ be two jets of affine plane curves at infinity. We call them \emph{logarithmically equivalent} if there are coordinates $(x_0,y_0)$ obtained from $(x,y)$ by a polynomial automorphism of $\mathbb A^2$, integers $D_I,D_{II}\in\mathbb Z_{>0}$, local parameters $s_I,s_{II}$, and series
\(
f_I,f_{II}\in\mathbb C((z))
\)
such that
\begin{equation}
\phi_i^*x_0=s_i^{-D_i},
\qquad
\phi_i^*y_0=f_i(s_i),
\qquad
i\in\{I,II\}.
\end{equation} Moreover, writing
$L:=\operatorname{lcm}(D_I,D_{II})$ and $\mu_L:=\{\zeta\in\mathbb C^\times\mid \zeta^L=1\}$, we require that, for some $\mu\in\mu_L$,
\[
f_I\big(z^{L/D_I}\big)-f_{II}\big((\mu z)^{L/D_{II}}\big)=O(z^L).
\]
When these conditions hold, we write
\(
(C_I,\phi_I,p_I)\sim(C_{II},\phi_{II},p_{II}).
\)
\end{definition}
\begin{definition}
For a given jet of an affine plane curve, a system of polynomial coordinates $(X,Y)$ is called \emph{good} if
\(
\operatorname{ord}\phi^*X<0.
\)
\end{definition}
In the definition of logarithmic equivalence, the two jets are compared in a common good coordinate system.

The geometric idea behind this definition is the following. Two jets of affine plane curves are equivalent if, in suitable coordinates, they both have the form $$\left(x, y(x)\right) \qquad y=a_{p-1} x^{\frac{p-1}{p}}+\ldots+a_{1-p}x^{-\frac{p-1}{p}}+O(1/x),$$ with the same coefficients, where we also identify expressions related by $x^{1/p}\mapsto\mu_px^{1/p}$ with $(\mu_p)^p=1$. This is precisely the definition in~\cite{Kon09}.

To show that logarithmic equivalence is an equivalence relation, we first establish the following technical lemmas.
\begin{lemma}
\label{lem:formal-estimates}
Let $p,q\in\mathbb C((u))$ satisfy
\[
p-q=O(u^N),
\qquad
\operatorname{ord}p=\operatorname{ord}q=r,
\qquad
N>r.
\]

\begin{enumerate}
\item
Let $\alpha\in\mathbb Q$, and suppose that $p^\alpha$ and $q^\alpha$ are defined in the field of Puiseux series, with compatible choices of roots of the leading coefficients. Then
\[
p^\alpha-q^\alpha
=
O\bigl(u^{N+(\alpha-1)r}\bigr).
\]

\item
Suppose instead that
\(
p,q\in u\mathbb C[[u]],\;
\operatorname{ord}p=\operatorname{ord}q=1,
\)
and let $p^{\operatorname{inv}}$ and $q^{\operatorname{inv}}$ be their compositional inverses. Then
\[
p^{\operatorname{inv}}-q^{\operatorname{inv}}
=
O(u^N).
\]
\end{enumerate}
\end{lemma}

\begin{proof}
For the first assertion, write
\[
q=p\cdot(1+f),
\qquad
f=\frac{q-p}{p}.
\]
Then
\(
\operatorname{ord}f\ge N-r>0.
\)
With compatible choices of roots, we have
\[
q^\alpha=p^\alpha\cdot(1+f)^\alpha.
\]
Since
\[
(1+f)^\alpha-1=O(f),
\]
we obtain
\[
q^\alpha-p^\alpha
=
p^\alpha O(f)
=
O\bigl(u^{\alpha r}\bigr)O\bigl(u^{N-r}\bigr)
=
O\bigl(u^{N+(\alpha-1)r}\bigr).
\]

For the second assertion, set
\(
P=p^{\operatorname{inv}},
\;
Q=q^{\operatorname{inv}}.
\)
We have
\[
0=p(P)-q(Q)
=
\bigl(p(P)-p(Q)\bigr)+\bigl(p(Q)-q(Q)\bigr).
\]
There is a series $A(X,Y)\in\mathbb C[[X,Y]]$ such that
\[
p(X)-p(Y)=(X-Y)A(X,Y),
\qquad
A(0,0)=p'(0)\ne0.
\]
Hence
\[
(P-Q)A(P,Q)
=
-\bigl(p(Q)-q(Q)\bigr).
\]
Since $p-q=O(u^N)$ and $\operatorname{ord}Q=1$,
\[
p(Q)-q(Q)=O(u^N).
\]
The series $A(P,Q)$ is a unit, so
\[
P-Q=O(u^N).
\]

\end{proof}

We next show that any two good coordinate systems can be joined by a chain of polynomial transformations along which the coordinate system remains good.

\begin{lemma}
\label{lem:good-coordinate-chain}
Let $(X,Y)$ and $(X',Y')$ be two good coordinate systems for a fixed jet of an affine plane curve at infinity. Then there is a chain of good coordinate systems
\[
(X,Y)=(X_0,Y_0),\ (X_1,Y_1),\ldots,\ (X_N,Y_N)=(X',Y')
\]
in which each transition has one of the following forms:
\[
\tag{E1}
(U,V)\longmapsto(U,aV+F(U)),
\quad
a\in\mathbb C^\times,
\quad
F\in\mathbb C[t],
\]
\[
\tag{E2}
(U,V)\longmapsto(aU+F(V),V),
\quad
a\in\mathbb C^\times,
\quad
F\in\mathbb C[t],
\quad
\operatorname{ord}\phi^*\bigl(aU+F(V)\bigr)<0,
\]
\[
\tag{S}
(U,V)\longmapsto(V,U),
\quad
\operatorname{ord}\phi^*V<0.
\]
\end{lemma}

\begin{proof}
By the van der Kulk theorem~\cite{Kul53}, the automorphism taking $(X,Y)$ to $(X',Y')$ factors into affine and elementary triangular automorphisms. We express these in terms of \textup{(E1)}, \textup{(E2)}, and~\textup{(S)}.

The map $(U,V)\mapsto(U,aV+F(U))$ already has the form \textup{(E1)}. The map $(U,V)\mapsto(aU+F(V),V)$ has the form \textup{(E2)} if its first coordinate has negative order. Otherwise $\operatorname{ord}V<0$, since the jet lies at infinity, and the required map can be written, up to interchanging the coordinates, as
$$(U,V)\mapsto(V,U)\mapsto(V,aU+F(V)).$$

It remains to consider an affine transformation $(U,V)\mapsto(P,Q)$, where
$P=\alpha U+\beta V+\varepsilon$, $Q=\gamma U+\delta V+\kappa$, and
$\Delta=\alpha\delta-\beta\gamma\ne0$. If
$\operatorname{ord}P<0$ and $\alpha\ne0$, use the chain
$(U,V)\mapsto(P,V)\mapsto(P,Q)$. It is easy to check that the second step has the form \textup{(E1)}.

If $\alpha=0$, then $P=\beta V+\varepsilon$ and
$\operatorname{ord}V<0$. We can use
$(U,V)\mapsto(V,U)\mapsto(P,U)\mapsto(P,Q)$; again the last step has the form \textup{(E1)}.

If $\operatorname{ord}P\ge0$, then
$\operatorname{ord}Q<0$, and the same construction applies to
$(Q,P)$.

At each stage of the factorization, we write the current coordinates either in the required order or in the opposite order. In the latter case, we conjugate the next factor by the coordinate interchange; it remains affine or elementary triangular. At the end we obtain either $(X',Y')$ or $(Y',X')$. In the second case, apply \textup{(S)}, which is allowed because $$\operatorname{ord}X'<0.$$
\end{proof}

We now check that each allowed transition preserves logarithmic equivalence.
\begin{lemma}\label{lem:elem_operations}
Suppose two jets of affine plane curves are equivalent in good coordinates $(U,V)$, and $(X,Y)$ is obtained by one of the transitions \textup{(E1)}, \textup{(E2)}, or~\textup{(S)} that is allowed for the first jet. Then $(X,Y)$ is good for the second jet as well, and the jets remain equivalent in these coordinates.
\end{lemma}

\begin{proof}
Choose parameters $s_i$ such that $U_i(s_i)=s_i^{-D_i}$. Set
$L=\operatorname{lcm}(D_1,D_2)$ and $k_i=L/D_i$. By definition, there is a
$\zeta\in\mu_L$ such that, with $\varepsilon_1=1$ and
$\varepsilon_2=\zeta$, we have
\[
A_i(t):=V_i\bigl((\varepsilon_i t)^{k_i}\bigr),\qquad i\in\{1,2\},
\qquad
A_1(t)-A_2(t)=O(t^L),
\]
and the same substitutions make both first coordinates equal to $t^{-L}$.

For \textup{(E1)}, the first coordinates remain $t^{-L}$, and
the difference of the second coordinates is $a(A_1-A_2)=O(t^L)$, since the term $F(t^{-L})$ is the same for both series.

Consider \textup{(S)}. We know that $\operatorname{ord}V_1<0$. Since $A_1(t)$ and $A_2(t)$ agree modulo $t^L$, we have $\operatorname{ord}_tA_1=\operatorname{ord}_tA_2=-Q<0$. If $q_i=-\operatorname{ord}_{s_i}V_i$, then $Q=k_iq_i$ and $Q=\operatorname{lcm}(q_1,q_2)$.

Choose roots of the leading coefficients compatibly, and set
$w_i=A_i^{-1/Q}$. Then $w_i\in t\mathbb C[[t]]$,
$\operatorname{ord}w_i=1$, and Lemma~\ref{lem:formal-estimates} gives
$w_1-w_2=O(t^{L+Q+1})$. The compositional inverses
$h_i=w_i^{\operatorname{inv}}$ satisfy the same estimate. Set $t=h_i(\tau)$. The new first coordinates are $\tau^{-Q}$, and the new second coordinates are $h_i(\tau)^{-L}$. Since $(L+Q+1)-L-1=Q$, we get \[
h_1(\tau)^{-L}-h_2(\tau)^{-L}=O(\tau^Q).
\]

For \textup{(E2)}, first suppose that
$\operatorname{ord}V_1<0$. The transition then factors as
\[
(U,V)\longmapsto(V,U)\longmapsto(V,aU+F(V))
\longmapsto(aU+F(V),V).
\]
Equivalence follows from the cases \textup{(E1)} and \textup{(S)} already considered.

Now suppose that $\operatorname{ord}V_1\ge0$. Then $\operatorname{ord}V_2\ge0$ as well, and
$A_1,A_2\in\mathbb C[[t]]$. After\[(U,V)\mapsto(aU+F(V),V),\] we have $X_i(t)=at^{-L}+F(A_i(t))$. Choose $c\in\mathbb C$ such that
$c^{-L}=a$, and
set
\[
w_i(t)=ct\bigl(1+a^{-1}t^LF(A_i(t))\bigr)^{-1/L},
\qquad X_i(t)=w_i(t)^{-L}.
\]
The estimate $F(A_1)-F(A_2)=O(t^L)$ implies
$w_1-w_2=O(t^{2L+1})$, so the inverses
$h_i=w_i^{\operatorname{inv}}$ also satisfy
$h_1(\tau)-h_2(\tau)=O(\tau^{2L+1})$. After the substitutions $t=h_i(\tau)$,
the first coordinates are $\tau^{-L}$, and
\[
\begin{aligned}
A_1(h_1)-A_2(h_2)
&=\bigl(A_1(h_1)-A_2(h_1)\bigr)
  +\bigl(A_2(h_1)-A_2(h_2)\bigr)\\
&=O(\tau^L).
\end{aligned}
\]
Indeed, the first term is $O(h_1^L)=O(\tau^L)$, while
$A_2(h_1)-A_2(h_2)$ is divisible by $h_1-h_2=O(\tau^{2L+1})$. Since
$A_2\in\mathbb C[[t]]$, this difference has order at least $2L+1$.

\end{proof}

\begin{corollary} \label{lem:coord_independence}
If two jets of affine plane curves are equivalent in some good polynomial coordinate system, then every polynomial coordinate system $(X,Y)$ that is good for one of them is good for the other, and the jets are equivalent in that system.
\end{corollary}

Reflexivity and symmetry are immediate from the definition. The following lemma proves transitivity and completes the verification that logarithmic equivalence is an equivalence relation.

\begin{lemma}\label{thm:J_transitivity}
Logarithmic equivalence of jets of affine plane curves is transitive.
\end{lemma}

\begin{proof}
Suppose $j_1\sim j_2$ and $j_2\sim j_3$. Choose coordinates $(x,y)$ that are good for the middle jet. By Corollary~\ref{lem:coord_independence}, they are good for all three jets, and both equivalences hold in these coordinates. Fix parameters $s_i$ and integers $D_i\in\mathbb Z_{>0}$ such that
$x_i=s_i^{-D_i}$,  $y_i=Y_i(s_i)$.

Set $L_{12}=\operatorname{lcm}(D_1,D_2)$,
$L_{23}=\operatorname{lcm}(D_2,D_3)$, and
$N=\operatorname{lcm}(L_{12},L_{23})$. Let $\zeta_{12}\in\mu_{L_{12}}$ and $\zeta_{23}\in\mu_{L_{23}}$ be the roots of unity appearing in the defining equalities. Choose $a,b\in\mathbb C^\times$ such that
$a^{N/L_{12}}=\zeta_{12}$ and $b^{N/L_{23}}=\zeta_{23}$. Then
$a,b\in\mu_N$, and the corresponding substitutions give
\[
\begin{aligned}
Y_1\bigl(t^{N/D_1}\bigr)-Y_2\bigl((at)^{N/D_2}\bigr)&=O(t^N),\\
Y_2\bigl(t^{N/D_2}\bigr)-Y_3\bigl((bt)^{N/D_3}\bigr)&=O(t^N).
\end{aligned}
\]
Replacing $t$ by $at$ in the second equality and summing them up, we obtain
\[
Y_1\bigl(t^{N/D_1}\bigr)-Y_3\bigl((bat)^{N/D_3}\bigr)=O(t^N).
\]
Set $L_{13}=\operatorname{lcm}(D_1,D_3)$, $m=N/L_{13}$, and
$\mu=(ba)^m\in\mu_{L_{13}}$. The last equality has the form
$G(t^m)=O(t^{mL_{13}})$, where
$G(z)=Y_1(z^{L_{13}/D_1})-Y_3((\mu z)^{L_{13}/D_3})$.
Hence $G(z)=O(z^{L_{13}})$, and $j_1\sim j_3$ by definition.

\end{proof}

\begin{definition}
The \emph{space of truncated jets of affine plane curves at infinity} is the quotient
\[
J^\infty
=
\{\text{jets of affine plane curves at infinity}\}\big/\!\sim,
\]
where $\sim$ is the equivalence relation of Definition~\ref{J_def}.
\end{definition}

This space gives a partial description of an elementary block. More precisely, if we forget the regular connection $R$ in each elementary block, the resulting objects correspond bijectively to the elements of $J^\infty$.

\begin{definition}
A \emph{minimal triple} at a point \(c\in\mathbb P^1\) is an equivalence class of data
\[
(\rho,[\varphi],c),\qquad
\rho\in u\mathbb C[[u]]\setminus\{0\},\qquad
[\varphi]\in\mathbb C((u))/\mathbb C[[u]],
\]
admitting no factorization
\[
\rho=\widetilde\rho\circ\sigma,\qquad
[\varphi]=[\widetilde\varphi\circ\sigma],
\]
where \(\sigma\in u\mathbb C[[u]]\) is a nonzero series of order
greater than \(1\), \(\widetilde\rho\in v\mathbb C[[v]]\), and
\(\widetilde\varphi\in\mathbb C((v))\).

Two triples at \(c\) are equivalent if there is an invertible formal change of parameter \(h\in u\mathbb C[[u]]\), with \(h'(0)\ne0\), such that
\[
\rho_1=\rho_2\circ h,\qquad
[\varphi_1]=[\varphi_2\circ h].
\]
\end{definition}

A choice of parameter puts \(\rho\) in the form \(u^p\), and the class \([\varphi]\) has a representative that is a polynomial in \(u^{-1}\) with zero constant term. Minimality then means that the greatest common divisor of \(p\) and the exponents of the nonzero terms of \(\varphi\) is~\(1\). In particular, \([\varphi]=0\) implies \(p=1\).

In this notation, define the correspondence
\begin{equation}\label{eq:j-El_corr}
\begin{aligned}
(u^p,[\varphi],c),\quad c\in\mathbb C,\ [\varphi]\ne0
&\longmapsto
\left[x=c+u^p,\ y=\frac{\varphi'}{p u^{p-1}}\right],\\
(u^p,[\varphi],\infty),\quad [\varphi]\ne0
&\longmapsto
\left[x=u^{-p},\ y=-\frac{u^{p+1}\varphi'}{p}\right],\\
(u, 0, c),\quad c\in\mathbb C
&\longmapsto [(c,u^{-1})],\\
(u, 0, \infty)
&\longmapsto [(u^{-1},0)].
\end{aligned}
\end{equation}
In the first two cases, \(\mathrm{d}\varphi=y\,\mathrm{d}x\).

\begin{propos}
\label{prop:minimal-triples-J}
The correspondence~\eqref{eq:j-El_corr} is a bijection between minimal triples and elements of \(J^\infty\).
\end{propos}

\begin{proof}
Start with a minimal triple. After choosing a suitable parameter, we may assume that \(\rho=u^p\) and that \(\varphi\) is a polynomial in \(u^{-1}\) with zero constant term.
The formulas~\eqref{eq:j-El_corr} give a rational parametrization in which at least one coordinate has a pole.
If \([\varphi]\ne0\), then \(\mathbb C(x)=\mathbb C(u^p)\).
The transformation \(u\mapsto\zeta u\), with \(\zeta^p=1\), preserves \(y\) if and only if \(\zeta^j=1\) for every degree \(j\) of a nonzero term of \(\varphi\). By minimality, this forces \(\zeta=1\), so
\(\mathbb C(x,y)=\mathbb C(u)\).
If \([\varphi]=0\), minimality gives \(p=1\), and the formulas are birational parametrizations of lines.
Different presentations of the same triple with \(\rho=u^p\) differ by \(u\mapsto\zeta u\), \(\zeta^p=1\), and define the same jet.

We now construct the inverse correspondence. First we recover the data \((\rho,[\varphi],c)\) without imposing minimality, and then choose a new parameter to obtain a minimal triple.

Integrating the terms of the $1$-form \(y\,\mathrm d x\) whose degrees are less than \(-1\), we obtain the unique class \([\varphi]\in\mathbb C((u))/\mathbb C[[u]]\) for which
\[
    y\,\mathrm d x-\mathrm d\varphi
    \in u^{-1}\mathbb C[[u]]\,\mathrm d u.
\]
This condition is preserved under formal changes of parameter and substitutions \(u=w^r\). By Corollary~\ref{lem:coord_independence}, equivalent jets can be compared in the coordinates \((x,y)\) if \(x\) has a pole, and in \((y,x)\) otherwise.

Suppose \(x\) has a pole. Set \(c=\infty\) and \(\rho=1/x\), and choose a parameter with \(x=u^{-p}\).
Then \(\rho=u^p\), and \([\varphi]\) is determined by the terms of \(y\) with degree less than \(p\).
For two equivalent representatives, the substitutions in the definition give
 \[x_1=x_2=z^{-L}, \quad y_1-y_2=O(z^L).\]
The forms \(y_i\,\mathrm d x_i\) differ by an element of \(z^{-1}\mathbb C[[z]]\,\mathrm d z\), so \([\varphi_1]=[\varphi_2]\); the series \(\rho_i=1/x_i\) also agree.
If \([\varphi]=0\), then \(y=O(u^p)\), and the jet is equivalent to that of a line. We associate with it the triple \((u,0,\infty)\).

Now suppose that \(x\) has no pole and \(y\) has a pole of order \(D>0\). Set \(c=x(0)\).
If \(x\equiv c\), we obtain the jet of the line \(x=c\) and the triple
\((u,0,c)\).
Otherwise, set \(\rho=x-c\) and choose a parameter such that \(\rho=u^p\), where \(p=\operatorname{ord}(x-c)>0\).
If \(D\le p\), the form \(y\,\mathrm d x\) has at most a simple pole, so \([\varphi]=0\).
Choose the parameter \(s=y^{-1/D}\). Then \(y=s^{-D}\) and \(x-c=O(s^D)\), so the jet is equivalent to that of the line \(x=c\). Again we take \((u,0,c)\).
If \(D>p\), the class \([\varphi]\) is nonzero.

%Let us check that the construction is independent of the representative. 
We check that the construction is independent of the representative. 
After the substitutions in the definition of equivalence, we have
\(y_1=y_2=z^{-L}\), \(x_1-x_2=O(z^L)\).
The values \(x_i(0)\) agree, and
\[
    y_1\,\mathrm d x_1-y_2\,\mathrm d x_2
    =z^{-L}\,\mathrm d(x_1-x_2)
    \in z^{-1}\mathbb C[[z]]\,\mathrm d z,
\]
which gives \([\varphi_1]=[\varphi_2]\).
If \(\rho_1=x_1-c=O(z^L)\), then \(\rho_2=O(z^L)\) as well, so both representatives give \((u,0,c)\).
Otherwise the \(\rho_i\) have the same order \(e<L\) and the same leading coefficient.
By Lemma~\ref{lem:formal-estimates}, there is an invertible change of parameter \(h\) such that
\[
    \rho_2\circ h=\rho_1,\qquad
    h(z)-z=O(z^{L-e+1}).
\]
Since \(\operatorname{ord}\varphi_2=e-L\), Taylor expansion gives
\(\varphi_2\circ h-\varphi_2 =O(z^{e-L-1+L-e+1})=O(1)\).
Hence \([\varphi_2\circ h]=[\varphi_1]\), and the triples are equivalent.

We now pass to a minimal presentation.
If \(\rho=u^p\) and \(\varphi=\sum_{j<0}b_j u^j\ne0\),
set \(m=\gcd\bigl(p,\{j:b_j\ne0\}\bigr)\).
With \(v=u^m\), we obtain
\[
    \widetilde\rho(v)=v^{p/m},\qquad
    \widetilde\varphi(v)=\sum_{j<0}b_j v^{j/m}.
\]
The powers of a variable $v$ in the above formula now have no common divisor greater than one. 
A substitution \(u=w^r\) multiplies \(p\), all the powers of \(\varphi\), and their greatest common divisor by \(r\), so it leaves the minimal triple unchanged.
Changing the choice of roots changes the minimal presentation only by multiplying the parameter by a root of unity.
Together with the invariance established above, this shows that the jet class determines a minimal triple.

We check that the two constructions are inverse to each other.
For a minimal triple with \([\varphi]\ne0\), the inverse construction recovers \(\rho\) as \(x-c\) or \(1/x\), and \([\varphi]\) from \(y\,\mathrm d x=\mathrm d\varphi\).

It remains to check that reconstructing a jet preserves its class.
If \(x=u^{-p}\), the reconstructed coordinate \(y\) differs from the original by \(O(u^p)\), which gives equivalent jets directly.

At a finite point, with \(D>p\), write \(x=c+u^p\), \(y=P+Q\),
where \(P\) contains the terms of degree \(<-p\) and \(Q=O(u^{-p})\).
The reconstructed pair is \((c+u^p,P)\).
Set
\[
    s_1=(P+Q)^{-1/D},\qquad s_2=P^{-1/D},
\]
with compatible choices of roots, and denote the inverse series by \(h_i\).
Lemma~\ref{lem:formal-estimates} gives
\(s_1-s_2=O(u^{D-p+1})\), whence
\[
    h_1-h_2=O(s^{D-p+1}),\qquad
    h_1^p-h_2^p=O(s^D).
\]
In these parameters both \(y\)-coordinates are \(s^{-D}\), and the difference of the \(x\)-coordinates is \(O(s^D)\). The jets are therefore equivalent.
Together with the cases \([\varphi]=0\) considered above, this proves that the two maps are inverse to each other and establishes the required bijection.
\end{proof}

The group $\operatorname{Aut}(\mathbb A^2)$ acts naturally on $J^\infty$: if $j$ is represented by $(x(u),y(u))$ and $g\in\operatorname{Aut}(\mathbb A^2)$, then $gj$ is represented by $g\left(x(u),y(u)\right)$.
A polynomial automorphism takes polynomial coordinates to polynomial coordinates, so it takes equivalent jets to equivalent jets. The action is therefore well defined.

\subsection{Connections of rank one}\label{sRk=1}
We recall the classification of formal meromorphic connections of rank one and identify the regular ones. We will use this description to construct the auxiliary connections \(H_j\).

Every formal meromorphic connection of rank one is isomorphic to $\big(\mathbb C((t)), \nabla\big)$, where $\nabla$ is determined by its value on $1\in \mathbb C$.

To determine when $(\mathbb C((t))_{1}, \nabla_1)$ and $(\mathbb C((t))_{2}, \nabla_2)$ are isomorphic, set $\nabla_1(1)=\phi_1$ and $\nabla_2(1)=\phi_2$. Denote the unit elements of $\mathbb C((t))_1$ and $\mathbb C((t))_2$ by $e_1$ and $e_2$, respectively.
Suppose there is an isomorphism of the vector spaces $\mathbb C((t))_{1}$ and $\mathbb C((t))_{2}$ sending $e_1$ to $\alpha e_{2}$, where $\alpha\in\mathbb C((t))^\times$.
Then
$$\nabla_1(e_1)=\nabla_2(\alpha e_{2})=\alpha'e_2+\alpha\phi_2e_2.$$
It follows that $\phi_1=\phi_2+\frac{\alpha'}{\alpha}$. Since $\alpha\in\mathbb C((t))^\times$, we have $$\frac{\alpha'}{\alpha}=\frac{n}{t}+\mathbb C[[t]], n\in\mathbb Z,$$ and every expression of this form occurs for some $\alpha$.
Thus the isomorphism classes of formal meromorphic connections of rank one are parametrized by $\mathbb C((t))/(\mathbb C[[t]]+\frac{\mathbb Z}{t}).$
Every isomorphism class has a representative with \begin{equation}\label{Enabla}\nabla(1)=p(t^{-1})+\frac{\lambda}{t},\hspace{10pt}p\in z^2\mathbb C[z], \lambda\in\mathbb C;\end{equation} here $p$ is uniquely determined by the isomorphism class, while $\lambda$ is determined modulo $\mathbb Z$.

A $\mathbb C((t))$-connection of rank one is \emph{regular} if $p=0$ in~\eqref{Enabla}, and \emph{irregular} otherwise. Thus regular rank-one connections are determined by $\lambda\in \mathbb C/\mathbb Z$. In terms of differential equations, their monodromy is multiplication by $\exp{(-2\pi i \lambda)}$.

\begin{definition}
To a one-dimensional formal connection $H=(\mathbb C((t)), \nabla_1)$ we associate its \emph{dual connection} $H^{\vee}=(\mathbb C((t)), \nabla_2)$, defined by $$\nabla_1(1)+\nabla_2(1)=0.$$ In particular, the dual of the trivial connection is trivial.
\end{definition}

\subsection{The decorated space \texorpdfstring{$\widetilde{J}^\infty$}{J-tilde infinity}}
A minimal triple does not retain the regular connection occurring in an elementary block. To include this information and obtain an object that is compatible with the automorphism actions and describes the whole block, we associate with each $j\in J^\infty$ a one-dimensional regular $\mathbb C ((t))$-connection $H_j$.

\begin{definition}
Let $[j]\in J^\infty$. Define $H_j$ from the minimal triple $(\rho, \varphi, c)$ corresponding to $[j]$ by

\begin{equation}
    H_j:=
\begin{cases}
\bigl(\mathbb C((u)),\mathrm d\bigr),
& j\longleftrightarrow(u, 0, c),\quad c\in\mathbb P^1,\\[3pt]
\displaystyle
\left(
\mathbb C((u)),
\mathrm d+\frac12\mathrm d\log\left(\frac{\mathrm d\rho}{\mathrm du}\right)
\right),
& j\longleftrightarrow(\rho,\varphi,c),\quad [\varphi]\ne0.
\end{cases}
\end{equation}
\end{definition}

\begin{lemma}\label{prop:H}
The isomorphism class of $H_j$ does not depend on the representative of the minimal triple $(\rho,\varphi,c)$.
\end{lemma}
\begin{proof}
Let $\widetilde\rho=\rho\circ h$, where
$h\in u\mathbb C[[u]]$ and $h'(0)\ne0$. Then
\[
    \mathrm{d}\log(\rho\circ h)'
    =
    \left(\mathrm{d}\log(\rho'\circ h)\right)+\operatorname{d}\log h'.
\]
Since $h'\in\mathbb C[[u]]^\times$, the $1$-form $\frac12\mathrm{d}\log h'$ equals $g(u)\,\mathrm{d}u$ with
$g(u)\in\mathbb C[[u]]$, so it does not change the isomorphism class of the connection. The class of $H_j$ is therefore independent of the representative.
\end{proof}
\begin{remark} \label{remH}
Let $p=\operatorname{ord}\rho$. Then
\(
    \rho'(u)=u^{p-1}a(u),\;
    a(u)\in\mathbb C[[u]]^\times.
\) Hence
\[
    H_j
    \simeq
    \left(
        \mathbb C((u)),
        \mathrm{d}+\frac{p-1}{2}\frac{\mathrm{d}u}{u}
    \right)
    =L_{1-p}
\] by~\eqref{eq:Lm}.
It follows that $H_j^\vee\simeq L_{p-1}$.
\end{remark}

\begin{definition}
Consider pairs
\[
(j,M_{j}),
\]
where \(j\in J^\infty\) and \(M_j\) is a regular formal connection over \(\mathbb C((u))\).

Two pairs $(j,M_j)$ and $(j',M_{j'})$ are equivalent if $j=j'$ in $J^\infty$ and
$M_{j}\simeq
M_{j'}$.
The quotient by this equivalence relation is called the \emph{decorated space of truncated jets of affine plane curves}. We denote it by $\widetilde{J}^\infty$ and write its elements as $(j,M_j)$.
\end{definition}

Define an action of $g\in\operatorname{Aut}(\mathbb A^2)$ on $\widetilde{J}^\infty$ by
\begin{equation}\label{eq:decorated-action}
g\cdot(j,M_j):=(gj,M_j).
\end{equation}

\section{Elementary blocks and elements of \texorpdfstring{$\widetilde{J}^\infty$}{J-tilde infinity}}
We now establish a one-to-one correspondence between elementary blocks and the decorated space of truncated jets of curves. We show that, under this correspondence, the action of Weyl algebra automorphisms agrees with the action of $\operatorname{Aut} \mathbb{A}^2$, and give another description of $H_j$.

\subsection{The correspondence for an elementary block}\label{SS:elbl}
With an elementary block $\operatorname{El}(\rho, \varphi, R)$ we associate the pair $(j,R\otimes H_j)\in\widetilde{J}^\infty$, where $j\in J^\infty$ is given by~\eqref{eq:j-El_corr}. By Proposition~\ref{prop:minimal-triples-J} and the definition of $\widetilde{J}^\infty$, this gives a bijection between isomorphism classes of elementary blocks and elements of $\widetilde{J}^\infty$.

Consider the subgroup
\[
\mathfrak G=
\left\{
(P,Q)\in\operatorname{Aut}\mathbb A^2
\ \middle|\
\det\frac{\partial(P,Q)}{\partial(x,y)}=1
\right\}.
\]
By~\cite[Theorem~2]{Mak84}, there is an isomorphism
\[
\operatorname{Aut}(W_{1,\mathbb C})
\xrightarrow{\sim}
\mathfrak G,
\qquad
\zeta\longmapsto g_\zeta.
\]
The construction of this isomorphism shows that, for \(P\in\mathbb C[x]\) and \(\alpha\in\mathbb C^\times\), the following maps correspond to one another:
\begin{equation}\label{eq:three-weyl-transformations}
\begin{aligned}
\sigma_P&:\quad x\mapsto x,\quad
    \partial_x\mapsto\partial_x+P(x),
&g_{\sigma_P}(x,y)&=(x,y+P(x)),\\
\xi_\alpha&:\quad x\mapsto\alpha x,\quad
    \partial_x\mapsto\alpha^{-1}\partial_x,
&g_{\xi_\alpha}(x,y)&=(\alpha x,\alpha^{-1}y),\\
\mathcal F&:\quad x\mapsto\partial_x,\quad
    \partial_x\mapsto-x,
&g_{\mathcal F}(x,y)&=(y,-x).
\end{aligned}
\end{equation}
These automorphisms generate \(\operatorname{Aut}(W_{1,\mathbb C})\) and \(\mathfrak G\), respectively. Thus it suffices to check compatibility of the actions for the transformations
\eqref{eq:three-weyl-transformations}.

\begin{theorem}\label{T0}
Let \(E=\operatorname{El}(\rho,\varphi,R)\) be an elementary block at \(c\in\mathbb P^1\), with associated pair \((j,R\otimes H_j)\), and let \(a\) be one of the transformations above. Suppose its action on the elementary block is written as \(aE=\widehat E=\operatorname{El}(\widehat\rho,\widehat\varphi,\widehat R)\), with associated pair \((\widehat j,\widehat R\otimes H_{\widehat j})\). Then
    \begin{equation}\label{eq:elementary-covariance}
        \widehat j=g_a j,
        \qquad
        \widehat R\otimes H_{\widehat j}
        \simeq R\otimes H_j.
    \end{equation}
In other words, the action of $\mathfrak G$ on $\widetilde{J}^\infty$ agrees with the action of Weyl algebra automorphisms on elementary blocks.
\end{theorem}

\begin{proof}
\noindent\emph{Triangular transformations.}
Let \(\mathcal P'=P\). By~\eqref{eq:sigma-finite-final} and~\eqref{eq:sigma-infinity-final}, the exponential factor changes to \(\varphi+\mathcal P(x)\), while \(\rho\) and \(R\) stay the same. Since \(\mathrm d\varphi=y\,\mathrm dx\), we have
\[
    \mathrm d\bigl(\varphi+\mathcal P(x)\bigr)
    =
    \bigl(y+P(x)\bigr)\,\mathrm dx.
\]
Thus the new truncated jet is obtained from the original one by \((x,y)\mapsto(x,y+P(x))\), so \(\widehat j=g_{\sigma_P}j\). Since \(\widehat\rho=\rho\), the definition of \(H_j\) gives \(H_{\widehat j}\simeq H_j\). Together with \(\widehat R\simeq R\), this proves
\eqref{eq:elementary-covariance}.

We treat separately the case in which a block at infinity becomes regular after adding \(\mathcal P(1/\rho)\). By Remark~\ref{rem:trans}, this can happen only when the original block corresponds to a pair \((j,R\otimes 1_\mathrm{d})\), where \(1_\mathrm{d}\) is the trivial connection. After the transformation,
\[
    \widehat{\operatorname{El}}
    =
    \operatorname{El}(u ,0,R)_\infty
    \longleftrightarrow
    \bigl([(u^{-1},0)],R\otimes 1_\mathrm{d}\bigr).
\]
On the other hand,
\[
    [g_{\sigma_P}j]
    =
    [(\rho^{-1},O(\rho))]
    =
    [(u^{-1},0)]
\]
in \(J^\infty\), while the connection \(\mathcal M=R \otimes 1_\mathrm{d}\) is unchanged.
Thus~\eqref{eq:elementary-covariance} holds in this case as well.

\smallskip
\noindent\emph{Scaling.}
By~\eqref{eq:scaling-finite-final} and
\eqref{eq:scaling-infinity-final}
we have \(\widehat\rho=\alpha\rho\) at a finite point and
\(\widehat\rho=\alpha^{-1}\rho\) at infinity; in both cases,
\(\widehat R\simeq R\). The equality
\[
    y\,\mathrm dx=(\alpha^{-1}y)\,\mathrm d(\alpha x)
\]
shows that the corresponding truncated jet is \(g_{\xi_\alpha}j\). Multiplying \(\rho\) by a nonzero constant does not change \(\mathrm d\log\rho'\), so \(H_{\widehat j}\simeq H_j\), and~\eqref{eq:elementary-covariance} follows again.

\smallskip
\noindent\emph{The Fourier transform.}
The formulas for the ramification and exponential factors in
\eqref{eq:fourier-finite-irregular-formulas-final},
\eqref{eq:fourier-infinity-greater-one-final},
\eqref{eq:fourier-infinity-less-one-final}, and
\eqref{eq:fourier-slope-one-data} are compatible with the coordinate change \((x,y)\mapsto(y,-x)\). Let us check this for \(\mathscr F^{(0,\infty)}\); the other cases are similar.

The truncated jet associated with an irregular block is
\[
    \operatorname{El}(\rho,\varphi,R)
    \longleftrightarrow
    \left(x=\rho,\quad y=\frac{\varphi'}{\rho'}\right).
\]
After the local Fourier transform, we obtain
\[
    \operatorname{El}
    \left(
        \widehat\rho=\frac{\rho'}{\varphi'},
        \widehat\varphi=\varphi-\frac{\rho}{\rho'}\varphi',
        \widehat R
    \right).
\]
The corresponding truncated jet is
\[
    \left[
        \left(
            \widehat x=\frac{1}{\widehat\rho}
            =\frac{\varphi'}{\rho'},
            \quad
            \widehat y
            =-\widehat\rho^{\,2}
              \frac{\widehat\varphi'}{\widehat\rho'}
            =-\rho
        \right)
    \right].
\]
On the other hand,
\[
    g_{\mathcal F}
    \left(
        x=\rho,\,
        y=\frac{\varphi'}{\rho'}
    \right)
    =
    \left(
        \widehat x=\frac{\varphi'}{\rho'},\,
        \widehat y=-\rho
    \right).
\]
Thus the actions on \(J^\infty\) and on elementary blocks agree.

It remains to check compatibility for the regular connections. Set
$ p=\operatorname{ord}\rho, \
\widehat p=\operatorname{ord}\widehat\rho.$
By Remark~\ref{remH},
\(
    H_j\simeq L_{1-p},
    \
    H_{\widehat j}^{\vee}\simeq L_{\widehat p-1},
\)
and hence
\begin{equation}\label{eq:H-ratio-fourier-clear}
    H_j\otimes H_{\widehat j}^{\vee}
    \simeq
    L_{\widehat p-p}.
\end{equation}
Also,
\(
    \widehat p-p
    =
    \operatorname{ord}\frac{\widehat\rho'}{\rho'}.
\)
Setting \(q=-\operatorname{ord}\varphi\) and using the formulas for \(\operatorname{ord}\widehat\rho\) from~\S\ref{s2.3}, we obtain
\[
\begin{array}{c|c|c}
    \text{original block}
    & \widehat p-p
    & H_j\otimes H_{\widehat j}^{\vee}\\ \hline
    c\in\mathbb C,\ [\varphi]\ne0
    & q
    & L_q\\
    c=\infty,\ q>p
    & q-2p
    & L_q\\
    c=\infty,\ 0<q<p
    & -q
    & L_q\\
    c=\infty,\ q=p,\ [\widehat\varphi]\ne0
    & -\widehat q
    & L_{\widehat q}
\end{array}
\]
Here we have used $L_m\simeq L_n \iff n\equiv m \mod 2$. Thus the resulting regular factors agree with those appearing in
\eqref{eq:fourier-finite-irregular-formulas-final},
\eqref{eq:fourier-infinity-greater-one-final},
\eqref{eq:fourier-infinity-less-one-final},
and~\eqref{eq:fourier-slope-one-irregular}, so~\eqref{eq:elementary-covariance} holds in every irregular case.

If one of the blocks is regular, the check is similar.

\end{proof}
\subsection{A pullback description of \texorpdfstring{\(H_j\)}{Hj}}
There is also an invariant description of the regular connection $H_j$. Set
\[
S_v:=\left(\mathbb C((v)),\mathrm d+\frac12\frac{\mathrm dv}{v}\right),
\qquad
S_v^\vee:=\left(\mathbb C((v)),\mathrm d-\frac12\frac{\mathrm dv}{v}\right).
\]

Let $z=\rho(u)$, and consider
\[
\rho^+S_z\otimes S_u^\vee.
\]
Under $\rho^+$, the form \(\mathrm dz/z\) pulls back to \(\mathrm d\rho/\rho=\mathrm d\log\rho\), so
\[
    \rho^+ S_z\otimes S_u^\vee
    =
    \left(
        \mathbb C((u)),
        \mathrm d+\frac12\mathrm d\log\frac{\rho(u)}{u}
    \right)\simeq H_j.
\]

\subsection{The correspondence for holonomic $W_{1,\mathbb C}$-modules}
We now put the preceding correspondences together.

\begin{definition}
Let \(\mathcal C_{\mathrm{form}}\) be the category whose objects are families
\[
    (N_c)_{c\in\mathbb P^1},
    \qquad
    N_c\in\operatorname{Hol}(\mathscr D_{t_c}^{\star}),
\]
such that $N_c=0$ for all but finitely many $c$. Morphisms are defined componentwise.
\end{definition}
\begin{definition}
Define the \emph{global vanishing cycles functor}
\[
    \Phi_{\mathrm{glob}}\colon
    \operatorname{Hol}(W_{1,\mathbb C})
    \longrightarrow
    \mathcal C_{\mathrm{form}},
    \qquad
    \mathcal M\longmapsto
    \bigl(\Phi(\mathcal M_c)\bigr)_{c\in\mathbb P^1}.
\]
On morphisms it acts componentwise: \(f\mapsto\bigl(\Phi(f_c)\bigr)_{c\in\mathbb P^1}\).
\end{definition}

Let \(\mathcal M\) be a holonomic \(W_{1,\mathbb C}\)-module, and let \(I\subset\mathbb P^1\) be its set of singular points together with $\infty$. For each \(c\in I\), apply the Turrittin--Levelt decomposition:
\begin{equation}\label{eq:module-to-equipped-data}
    \Phi_{\mathrm{glob}}(\mathcal M)
    =
    \bigl(\Phi(\mathcal M_c)\bigr)_{c\in I}=\bigl(\bigoplus_{a\in A_c}
    \operatorname{El}
    (\rho_{c,a},\varphi_{c,a},R_{c,a})_c\bigr)_{c\in I},
\end{equation}
where \(A_c\) is a finite set indexing the elementary blocks in the decomposition of the formal meromorphic connection \(\Phi(\mathcal M_c)\).
Each summand gives an element of $\widetilde J^\infty$:
\[
    \bigl(j_{c,a},\mathcal N_{c,a}\bigr),
    \qquad
    \mathcal N_{c,a}
    :=
    R_{c,a}\otimes H_{j_{c,a}}.
\]
If several summands give the same jet, take the direct sum of the corresponding connections. In this way, \(\mathcal M\) determines a finite family of elements of $\widetilde J^\infty$.

Let \(\mathfrak R\) be the set of isomorphism classes of regular formal connections, modulo invertible changes of the formal variable, and put \(\mathfrak R_0:=\mathfrak R\sqcup\{0\}\). Define the finitely supported function
\begin{equation}\label{eq:nu-equipped}
\widetilde\nu_{\mathcal M}\colon J^\infty\longrightarrow\mathfrak R_0,
\qquad
\widetilde\nu_{\mathcal M}(j)
:=
\left[
\bigoplus_{\substack{c\in I,\ a\in A_c\\j_{c, a}=j}}
\mathcal N_{c, a}
\right].
\end{equation}
If no summand has jet \(j\), set \(\widetilde\nu_{\mathcal M}(j)=0\). The jet \(j\) already determines \(c\), so \(\widetilde\nu_{\mathcal M}\) retains the decomposition of the data by singular point.

\begin{theorem}\label{T1}
The function \(\widetilde\nu_{\mathcal M}\) uniquely determines the family \(\{\Phi(\mathcal M_c)\}_{c\in I}\). In particular, the elements of $\widetilde{J}^\infty$ parametrize elementary blocks at the various points $c\in\mathbb P^1$, up to isomorphism.
\end{theorem}
\begin{proof} For each \(j\), choose a representative \(\mathcal N_j\) of the class \(\widetilde\nu_{\mathcal M}(j)\). The regular factor of the corresponding elementary block is then
\[
R_j\simeq\mathcal N_j\otimes H_j^\vee,
\]
and its minimal triple is recovered from \(j\) by Proposition~\ref{prop:minimal-triples-J}. Conversely, the decomposition~\eqref{eq:module-to-equipped-data} directly determines \(\widetilde\nu_{\mathcal M}\). Thus decomposition into elementary blocks and passage to elements of $\widetilde J^\infty$  are inverse constructions on isomorphism classes of formal local data.
\end{proof}
The construction in Theorem~\ref{T1} does not, in general, recover the global module \(\mathcal M\) itself.

Let the automorphism \(a\in\operatorname{Aut}W_{1,\mathbb C}\) correspond to the polynomial symplectic automorphism
\(g_a\in\operatorname{Aut}\mathbb A^2\).
Its action on finitely supported functions is given by
\[
(g_a\widetilde\nu)(j):=\widetilde\nu(g_a^{-1}j).
\]
The compatibility established for elementary blocks immediately gives the following.
\begin{theorem}\label{T2} We have
\begin{equation}\label{eq:nu-covariance}
\widetilde\nu_{a(\mathcal M)}
=g_a\widetilde\nu_{\mathcal M}.
\end{equation}
\end{theorem}
In particular, restricting the right-hand side to jets corresponding to triples with \([\varphi]\ne0\) determines the irregular parts of all formal localizations of \(a(\mathcal M)\).

Forgetting the isomorphism class of the regular connection and retaining only its rank gives the finitely supported multiplicity function
\begin{equation}\label{eq:nu-rank}
\nu_{\mathcal M}\colon J^\infty\longrightarrow\mathbb Z_{\ge0},
\qquad
\nu_{\mathcal M}(j)
:=
\operatorname{rk}\widetilde\nu_{\mathcal M}(j),
\end{equation}
where the zero value is assigned rank zero.

\begin{definition}
Two holonomic modules belong to the same \emph{logarithmic family} if their multiplicity functions \(\nu\) agree. If their functions \(\widetilde\nu\) agree, they belong to the same \emph{decorated logarithmic family}.
\end{definition}

\begin{remark}
The function \(\nu\) is the rank version of the decorated invariant \(\widetilde\nu\); it corresponds to the multiplicity function used to describe logarithmic families in~\cite{Kon09}. Thus \(\widetilde\nu\) refines this invariant by retaining the isomorphism class of the regular formal connection as well as its rank.
\end{remark}

\begin{example}
Let us compute $\tilde \nu$ for the module $\mathcal M$ defined by the following generalization of the Airy equation: \begin{equation}
        \partial^ny+a_{n-1}\partial^{n-1}y+\ldots+a_0y-bxy=0,\qquad a_i \in\mathbb C, \ b\in\mathbb{C}^\times. \label{eq:Big_Airy}
\end{equation} This equation is singular only at infinity. Its germ \(\mathcal M_\infty\) is already a formal meromorphic connection, so formalization preserves it up to natural isomorphism.

The inverse Fourier transform $\mathcal F^{-1}\mathcal M$ corresponds to the equation $$(-1)^nx^ny +(-1)^{n-1}a_{n-1}x^{n-1}y+\ldots+a_0y-b\partial y=0.$$ It is also singular only at infinity, and its space of formal solutions there is one-dimensional: $$y=c\cdot\exp\big(\frac{1}{b}\int [(-1)^nx^n +(-1)^{n-1}a_{n-1}x^{n-1}+\ldots+a_0]\mathrm{d}x \big), \ c\in \mathbb C.$$
Thus $\Phi(\mathcal F^{-1}\mathcal M)_\infty$ is described by a single elementary block: \begin{equation}
        \operatorname{El}\big(u, \frac{(-1)^n}{b}\left(\frac{u^{-n-1}}{n+1} -a_{n-1}\frac{u^{-n}}{n}+\ldots+(-1)^na_0u^{-1}\right),\ \left(\mathbb C((u)),\ \mathrm{d}\right)\big).
        \label{eq:El_FAiry}
    \end{equation}

Let $j_\mathcal F\in J^\infty$ be the jet with parametrization
    $$(x=u^{-1}, \ y=\frac{(-1)^n}{b}\left({u^{-n}}-a_{n-1}{u^{-n+1}}+\ldots+(-1)^na_0\right)).$$
Since $1_\mathrm{d}\otimes1_\mathrm{d}\simeq1_\mathrm{d}$ for trivial connections, the elementary block~\eqref{eq:El_FAiry} corresponds to the element of $\widetilde{J}^\infty$ given by $$(j_{\mathcal{ F}},\ 1_\mathrm{d}).$$ The function $\widetilde\nu$ is supported at this one point, where its value is the trivial connection.

Now apply the Fourier transform to return to~\eqref{eq:Big_Airy}. The function $\tilde\nu$ is nonzero only at the jet $j\in J^\infty$ with parametrization $$(x=\frac{(-1)^n}{b}\left({u^{-n}}-a_{n-1}{u^{-n+1}}+\ldots+(-1)^na_0\right),\ y=-u^{-1}),$$ and its value there is the trivial connection.
\end{example}

\appendix
\section{Formal power series and \texorpdfstring{$J^\infty_+$}{J-plus infinity}}
\label{sec:J-infinity-Pre}
Recall that $\mathbb C[[u]]$ is the ring of formal power series in $u$ and $\mathbb C((u))$ is the field of formal Laurent series in $u$, the fraction field of $\mathbb C[[u]]$.
Write $\mathcal G_\infty$ for the set of pairs of formal Laurent series.

For $f\in\mathbb C((u))$, set
$$f(0):=\begin{cases}\text{the coefficient of }u^0,&\text{if }f\in\mathbb C[[u]],\\\infty,&\text{otherwise}.\end{cases}$$
Here $\infty$ is a formal symbol not belonging to $\mathbb C$.

For $f=\sum_{i\in\mathbb Z} f_i u^i$, set $${\rm ord}(f):=\min_i\{i\mid f_i\ne 0\}.$$
Define $$O(f):=\{g\in\mathbb C((u))\mid {\rm ord}(g)\ge{\rm ord}(f)\}.$$

We want to define $J^\infty_+$ as the quotient of $\mathcal G_\infty$ by an equivalence relation $\sim$.
We specify $\sim$ as a subset of $\mathcal G_\infty\times\mathcal G_\infty$, that is, as a relation on $\mathcal G_\infty$. 
Recall that a relation on a set $X$ is simply a subset of $X\times X$.
\begin{definition}
Define $\sim$ as the union of three relations on $\mathcal G_\infty$:
\[\sim_{\mathrm{fin}},\qquad \sim_{x},\qquad \sim_{y}.\]
They are given as follows.

$\sim_{\mathrm{fin}}:$ $(x_1(u),y_1(u))\sim_{\mathrm{fin}}(x_2(u),y_2(u))$ if and only if $$x_1(0)=x_2(0)\ne\infty,\hspace{10pt} y_1(0)=y_2(0)\ne\infty;$$

$\sim_x:$ $(x_1(u),y_1(u))\sim_x(x_2(u),y_2(u))$ if and only if $$x_1(0)=x_2(0)=\infty$$ and there exist $s_1,s_2\in u\mathbb C[[u]]$, $D\in\mathbb Z_{>0}$, and $f\in\mathbb C((u))$ such that
$$x_1(u)=s_1^{-D}, x_2(u)=s_2^{-D},\hspace{10pt} y_1(u)=f(s_1)+O(s_1^D), y_2(u)=f(s_2)+O(s_2^D);$$

$\sim_y:$ the image of $\sim_x$ under the coordinate interchange $(x,y)\mapsto(y,x)$, acting on both factors of $\mathcal G_\infty\times\mathcal G_\infty$.
\end{definition}
Using the Puiseux series $\tilde f(u):=f(u^{1/D})$, one can think of $\sim_x$ as follows: $$y_1(u)=\tilde f(\frac{1}{x_1(u)})+O(\frac{1}{x_1(u)}),\hspace{10pt} y_2(u)=\tilde f(\frac{1}{x_2(u)})+O(\frac{1}{x_2(u)}).$$
The interpretation of $\sim_y$ is similar.

\begin{lemma}\label{L:cass1}
The relation $\sim$ is an equivalence relation.
\end{lemma}
\begin{proof}

We must verify reflexivity, symmetry, and transitivity. The first two are immediate from the definition, so only transitivity needs checking.

The relation $\sim_{\operatorname{fin}}$ is clearly transitive.
For $\sim_x$, use the calculation in the proof of Lemma~\ref{thm:J_transitivity}: pass to the common exponent $N=\operatorname{lcm}(D,E)$ and observe that the common leading part is a series in $s^{N/\gcd(D,E)}$.
Transitivity of $\sim_y$ follows by the symmetry $(x,y)\mapsto(y,x)$.
A chain in which $\sim_{\operatorname{fin}}$ is adjacent to $\sim_x$ or $\sim_y$ is impossible: the middle element would have to satisfy both $x(0)\ne\infty$ and $x(0)=\infty$, or the corresponding conditions on $y$. This leaves the mixed case $a\sim_x b\sim_y c$; the case $a\sim_y b\sim_x c$ is symmetric. Here $y_b(0)=\infty$, and the argument for case~(S) in Lemma~\ref{lem:elem_operations}, together with the estimates of Lemma~\ref{lem:formal-estimates}, gives
\begin{equation}\label{Eari}
\qquad a\sim_x b,\ y_b(0)=\infty\ \Longrightarrow\ y_a(0)=\infty\ \text{and}\ a\sim_y b,
\end{equation}
so $a\sim_y c$ by transitivity of $\sim_y$. \end{proof}
We define
\[J^{\infty}_+\coloneqq \mathcal{G}_{\infty} / \sim .\]
Thus $J^\infty_+$ is, by definition, the set of equivalence classes of pairs of formal series for this logarithmic relation.

Let $\phi$ be an automorphism of $\mathbb A^2$, that is, a pair of polynomials $f,g\in\mathbb C[x,y]$ with $\mathbb C[x,y]=\mathbb C[f,g]$. Define its action on $\mathcal G_\infty$ by
$$\phi.(x_1(u), y_1(u)):=(f(x_1(u), y_1(u)), g(x_1(u), y_1(u))).$$
This induces a diagonal action on $\mathcal G_\infty\times\mathcal G_\infty$.
\begin{lemma}\label{L:cass2} The action of ${\rm Aut}(\mathbb A^2)$ preserves $\sim$.
\end{lemma}
\begin{proof}

By the Jung--van der Kulk theorem~\cite{Kul53, Jung42}, the group $\operatorname{Aut}(\mathbb A^2)$ is generated by $\sigma(x,y)=(y,x)$ and triangular automorphisms $T(x,y)=(x,ay+F(x))$, where $a\in\mathbb C^\times$ and $F\in\mathbb C[t]$; compare the proof of Lemma~\ref{lem:good-coordinate-chain}. Their inverses have the same form.
Since $\phi.(\psi.p)=(\phi\circ\psi).p$, it is enough to check that $\sigma$ and $T$ preserve $\sim$. For $\sigma$, this follows from the definition of $\sim_y$. The map $T$ preserves $\sim_{\operatorname{fin}}$, since it takes constant terms to
$(x(0),\,ay(0)+F(x(0)))$.
It also preserves $\sim_x$, just as in case~(E1) of Lemma~\ref{lem:elem_operations}: replace $f$ by $af(v)+F(v^{-D})$.
Now suppose $(x_1,y_1)\sim_y(x_2,y_2)$. If $x_1(0)=\infty$, then~\eqref{Eari}, with the coordinates interchanged, makes the pairs $\sim_x$-equivalent, and this case is already settled. If $x_1(0)\ne\infty$, then $x_1,x_2\in\C[[u]]$, and we are in case~(E2), $\ord V_1\ge0$, of the proof of Lemma~\ref{lem:elem_operations}.
Choose $c$ with $c^{-D}=a$ and set
$\sigma_i=c\,s_i\bigl(1+a^{-1}s_i^DF(x_i)\bigr)^{-1/D}$; then
$ay_i+F(x_i)=\sigma_i^{-D}$. Also set
$w(v)=c\,v\bigl(1+a^{-1}v^DF(f(v))\bigr)^{-1/D}$ and $h=w^{inv}$. The estimates and factorization in the proof of Lemma~\ref{lem:elem_operations} give
$x_i=(f\circ h)(\sigma_i)+O(\sigma_i^D)$, so the image pairs are $\sim_y$-equivalent.
\end{proof}
\begin{corollary}
There is a natural action of ${\rm Aut}(\mathbb A^2)$ on $J^\infty_+$.
\end{corollary}
To $(x(u),y(u))\in\mathcal G_\infty$ associate $(x(0),y(0))$. This defines a map
$$\mathcal{G}_\infty\to (\mathbb C\cup\{\infty\})\times(\mathbb C\cup\{\infty\}).$$
For each $(a,b)\in\mathbb C^2$, its fiber is a single $\sim$-equivalence class in $\mathcal G_\infty$. Denote the corresponding element of $J^\infty_+$ by $j_{(a,b)}$.

It is easy to see that $$J^{\infty}_+\cong J^\infty\sqcup_{(a, b)\in\mathbb C^2}j_{(a, b)}\cong J^\infty\sqcup\mathbb A^2.$$
This lets us view $J^\infty_+$ as a kind of `partial compactification' of $\mathbb A^2$, with $J^\infty$ as its divisor at infinity. The phrase is meant informally: the construction is not a compactification in either the algebraic or the topological sense.

\section{The regular part of a formal germ}
\label{app:regular-germ}

Elementary blocks describe formal meromorphic connections and hence determine the purely irregular part
$\mathcal M_{c,\operatorname{irr}}$.
The regular part $\mathcal M_{c,\operatorname{reg}}$ needs a separate description, since its meromorphic localization does not, in general, determine the original germ. We describe a regular germ by a quiver representation and explain the computations of $\Phi$ in~\S\ref{s2.2}.

Fix a finite point $c\in\mathbb C$, set $t=t_c$, and let $N$ be a regular holonomic $\mathscr D_t$-module. The category of such modules is equivalent to the category of finite-dimensional quiver representations $(E,F,u,v)$, where $u\colon E\to F$, $v\colon F\to E$, and $\operatorname{Id}_E+vu$ and $\operatorname{Id}_F+uv$ are invertible~\cite{Zap15}.

\begin{remark}
The categories of formal and analytic germs of regular holonomic $D$-modules are equivalent to each other and to the category of germs of perverse sheaves on a complex disk that are local systems away from the center.
\end{remark}

To construct the representation, use the increasing $V$-filtration. On $\mathscr D_t$ it is defined by
\[
    \sum_{i=0}^{r}a_i(t)\partial_t^i\in V_k\mathscr D_t
    \quad\Longleftrightarrow\quad
    i-\operatorname{ord}_t a_i\leq k
\quad\text{for all }i\text{ with }a_i\ne0.
\]
Set $\Sigma=\{\lambda\in\mathbb C\mid
0\leq\operatorname{Re}\lambda<1\}$, and choose on $N$ the unique good $V$-filtration normalized by
$\operatorname{Spec}(t\partial_t+k\mid\operatorname{gr}^V_kN)
\subset\Sigma-1$, where
$\operatorname{gr}^V_kN=V_kN/V_{k-1}N$.
Then $E=\operatorname{gr}^V_{-1}N$ and $F=\operatorname{gr}^V_0N$. The actions of $\partial_t$ and $t$ induce $a=\partial_t\colon E\to F$ and $b=t\colon F\to E$.
The operators $ba=t\partial_t$ on $E$ and $ab=\partial_t t$ on $F$ have spectra in~$\Sigma$.

The representation $(E,F,a,b)$ determines $N$. Following~\cite[Theorem~3.4]{Zap15}, define the arrows $\operatorname{can}$ and $\operatorname{var}$ in the quiver description by
\begin{equation}
\label{eq:app-quiver-normalization}
\begin{aligned}
    u:=\operatorname{can}&=a,\qquad
    v:=\operatorname{var}=b\,h(ab)=h(ba)\,b,\\
    h(s)&=\sum_{\ell=1}^{\infty}
    \frac{(-2\pi i)^\ell}{\ell!}s^{\ell-1}.
\end{aligned}
\end{equation}
For $s\ne0$, we have $h(s)=(e^{-2\pi i s}-1)/s$, and
$h(0)=-2\pi i$. The identity $s h(s)=e^{-2\pi i s}-1$ gives
\[
    \operatorname{Id}_E+vu=e^{-2\pi i ba},\qquad
    \operatorname{Id}_F+uv=e^{-2\pi i ab}.
\]
In particular, both operators are invertible, so $(E,F,u,v)$ is a representation of the quiver with two vertices $E$ and $F$ and arrows $u\colon E\to F$ and $v\colon F\to E$, subject to the invertibility conditions stated above.

In terms of $(E,F,a,b)$, the connections $\Psi(N)$ and $\Phi(N)$ are
\begin{equation}
\label{eq:app-psi-phi-quiver}
\begin{aligned}
    \Psi(N)&\simeq
    \left(\mathbb C((t))\otimes_{\mathbb C}E,
          \mathrm d+ba\,\frac{\mathrm dt}{t}\right),\\
    \Phi(N)&\simeq
    \left(\mathbb C((t))\otimes_{\mathbb C}F,
          \mathrm d+ab\,\frac{\mathrm dt}{t}\right).
\end{aligned}
\end{equation}
Under these identifications, $u$ and $v$ extend to morphisms of connections
\[
    \operatorname{can}:=1\otimes u\colon
    \Psi(N)\longrightarrow\Phi(N),
    \qquad
    \operatorname{var}:=1\otimes v\colon
    \Phi(N)\longrightarrow\Psi(N).
\]
Indeed, $ab\,u=u\,ba$ and $ba\,v=v\,ab$, so these maps respect the connections in~\eqref{eq:app-psi-phi-quiver}.
For $N=\mathcal M_{c,\operatorname{reg}}$, we therefore obtain the diagram
\begin{equation}
    \Psi(\mathcal M_{c,\operatorname{reg}})
    \xrightarrow{\operatorname{can}}
    \Phi(\mathcal M_{c,\operatorname{reg}})
    \xrightarrow{\operatorname{var}}
    \Psi(\mathcal M_{c,\operatorname{reg}}).
\end{equation}
For the analytic connections given by~\eqref{eq:app-psi-phi-quiver}, the monodromy along a positively oriented loop around $0$ is represented on $E$ and $F$, respectively, by
\[
    T_\Psi=e^{-2\pi i ba}=\operatorname{Id}_E+vu,\qquad
    T_\Phi=e^{-2\pi i ab}=\operatorname{Id}_F+uv.
\]
In other words, in these models the monodromy formulas are
$T_\Psi=\operatorname{id}+\operatorname{var}\circ\operatorname{can}$ and
$T_\Phi=\operatorname{id}+\operatorname{can}\circ\operatorname{var}$.
In particular, if the Jordan form of $ab$ is a single block
of size $n$ with eigenvalue $\mu$, then $\Phi(N)\simeq
\mathscr D_t^\star/\mathscr D_t^\star(t\partial_t-\mu)^n$.
This recovers the table for $\Phi$ in~\S\ref{s2.2} directly from the representations of the indecomposable modules.

All indecomposable regular $\mathscr D_t$-modules are listed below. Every regular holonomic module is a finite direct sum of these.
Here $n\ge1$, $m\ge0$, $\mu\in\Sigma\setminus\{0\}$,
$S_{\mu,n}$ is the Jordan block of size $n$ with eigenvalue $\mu$,
$P_m=(\operatorname{Id}_m,0_{m\times1})$, and
$J_m=\begin{pmatrix}0_{1\times m}\\\operatorname{Id}_m\end{pmatrix}$;
for $m=0$, the corresponding maps are zero.
The last column is computed using~\eqref{eq:app-psi-phi-quiver} and reproduces the table in~\S\ref{s2.2}.

\begin{center}
\begingroup
\renewcommand{\arraystretch}{2.5}
\begin{tabular}{c|c|c}
$N$ & $(E,F,a,b)$ & $\Phi(N)$ \\
\hline
$\displaystyle\frac{\mathscr D_t}
 {\mathscr D_t(t\partial_t-\mu)^n}$
& $(\mathbb C^n,\mathbb C^n,\operatorname{Id}_n,S_{\mu,n})$
& $\displaystyle\frac{\mathscr D_t^\star}
 {\mathscr D_t^\star(t\partial_t-\mu)^n}$ \\
$\displaystyle\frac{\mathscr D_t}{\mathscr D_t(t\partial_t)^n}$
& $(\mathbb C^n,\mathbb C^n,\operatorname{Id}_n,S_{0,n})$
& $\displaystyle\frac{\mathscr D_t^\star}
 {\mathscr D_t^\star(t\partial_t)^n}$ \\
$\displaystyle\frac{\mathscr D_t}{\mathscr D_t(\partial_t t)^n}$
& $(\mathbb C^n,\mathbb C^n,S_{0,n},\operatorname{Id}_n)$
& $\displaystyle\frac{\mathscr D_t^\star}
 {\mathscr D_t^\star(t\partial_t)^n}$ \\
$\displaystyle\frac{\mathscr D_t}
 {\mathscr D_t\bigl((\partial_t t)^m\partial_t\bigr)}$
& $(\mathbb C^{m+1},\mathbb C^m,P_m,J_m)$
& $\displaystyle\frac{\mathscr D_t^\star}
 {\mathscr D_t^\star(t\partial_t)^m}$ \\
$\displaystyle\frac{\mathscr D_t}
 {\mathscr D_t\bigl((t\partial_t)^m t\bigr)}$
& $(\mathbb C^m,\mathbb C^{m+1},J_m,P_m)$
& $\displaystyle\frac{\mathscr D_t^\star}
 {\mathscr D_t^\star(t\partial_t)^{m+1}}$
\end{tabular}
\endgroup
\end{center}

\end{document}